\documentclass[12pt]{amsart}
\usepackage{amssymb,latexsym,amsmath,amsthm,amscd}
\usepackage{stmaryrd}
\usepackage{setspace}
\usepackage{array}
\usepackage{float}
\usepackage[all]{xy} \xyoption{arc}
\usepackage[left=3cm,top=2cm,right=3cm,bottom = 2cm]{geometry}
\usepackage{graphicx}
\usepackage[usenames,dvipsnames]{xcolor}
\usepackage{subfig}
\usepackage{charter}
\usepackage{hyperref}

\theoremstyle{plain}
\newtheorem{thm}{Theorem}
\newtheorem{lem}[thm]{Lemma}
\newtheorem{prop}[thm]{Proposition}
\newtheorem{cor}[thm]{Corollary}

\theoremstyle{definition}

\theoremstyle{remark}
\newtheorem{rmk}[thm]{Remark}

\DeclareMathOperator{\Tr}{Tr}

\DeclareMathOperator{\SL}{SL}

\DeclareMathOperator{\id}{id}

\newcommand{\CC}{\mathbf{C}}

\newcommand{\ZZ}{\mathbf{Z}}

\DeclareMathOperator{\lcm}{lcm}

\DeclareMathOperator{\End}{End}
\DeclareMathOperator{\Inv}{Inv}

\begin{document}
\title[$1$-point functions for symmetrized Heisenberg and symmetrized lattice VOAs]{$1$-point functions for symmetrized Heisenberg and symmetrized lattice vertex operator algebras}
\author{Geoffrey Mason}
\email{gem@ucsc.edu}
\address{University of California Santa Cruz}
\author{Michael H. Mertens}
\email{mmertens@math.uni-koeln.de}
\address{Universit\"at zu K\"oln (on leave from Universit\"at Bielefeld)}
\thanks{The authors were supported by grants  $\# 62524$ and $427007$ respectively from the Simons Foundation.}
\date{4/4/22}

\begin{abstract}
We obtain explicit formulas for the $1$-point functions of all states in the symmetrized Heisenberg algebra $M^+$ and symmetrized lattice VOAs $V_L^+$. For this we employ a new 
$\ZZ_2$-twisted variant of so-called Zhu recursion.
\end{abstract}

\maketitle

\tableofcontents

\section{Introduction} 
Among the most prominent vertex operator algebras (VOAs) are the Heisenberg VOAs $M$ (of any rank, say $k$) and lattice theories $V_L$ for an even lattice $L$.\
Each of these theories admit certain natural involutive automorphisms and the corresponding fixed-point subVOAs are usually denoted by $M^+$ and $V_L^+$ and called the \textit{symmetrized Heisenberg} or \textit{symmetrized lattice VOA} respectively.\ (Additional background and details are provided in Section \ref{S2}.)\ The purpose of the present paper is to describe the $1$-point trace functions for the VOAs $M^+$ and $V_L^+$.

\medskip
We now give a more detailed overview of the contents and main results of the present paper.\ Since its very inception as a mathematical theory,  VOAs have been ineluctably linked with the theory of modular forms.\ The catchphrase is \textit{modular invariance in CFT}.\ Modular invariance is visibly at work in the
theory of $n$-point functions \cite{Z} of a VOA $V$.\ In particular the $1$-point functions, defined by
\begin{eqnarray*}
Z_V(u, \tau) :=\Tr_V o(u)q^{L(0)-c/24}=q^{-c/24}\sum_n \Tr_{V_n} o(u)q^n, \quad (q=e^{2\pi i\tau})
\end{eqnarray*}
where $V=\bigoplus V_n$ is a VOA of central charge $c$ and $u\in V$ a state of $V$ with zero mode $o(u)$ (cf.\ Subsections \ref{SS2.2}, \ref{SS2.3} for definitions and explanations of this terminology), are closely related to elliptic modular forms.\ There are several levels at which one can investigate the modular invariance of $1$-point functions for $V$:
\begin{eqnarray*}
&&(a)\ \mbox{Determine whether or not all $1$-point functions are (linear combinations of)}\\
&&\ \ \ \ \ \mbox{ elliptic modular forms.}\\
&&(b)\ \mbox{Determine the space of functions spanned by all $1$-point functions.}\\
&&(c)\ \mbox{Give explicit formulas for the $1$-point functions $Z_V(u_i, \tau)$ for a basis $\{u_i\}$ of $V$.}
\end{eqnarray*}

Problem (a) is important for many reasons, including applications to the structure theory of VOAs.\ One would like to know the precise class of VOAs enjoying the property that all $1$-point functions
are elliptic modular forms on a congruence subgroup.\  Presumably it includes the class of strongly regular VOAs (cf.\ \cite{M}) though this has not been proved.\ Sufficient conditions are known from the combined works of many authors, of whom we mention Zhu \cite{Z}, Huang \cite{H}, Dong-Ng \cite{DN} and  Dong-Li-Mason \cite{DLM}, although we will not need these results.\  In practice
even if we know that the answer to Problem (a) is affirmative for $V$, it offers little more than guidance when  trying to solve Problems (b) or (c).\ Problem (b) was answered for the Moonshine module $V^{\natural}$ in \cite{DM} and some other holomorphic VOAs in \cite{KH, KH1}.\
There are many interesting VOAs where not all $1$-point functions are modular yet they have interesting answers to Problems (b) and (c).\ Solving Problem (c) is challenging for any VOA $V$ and usually much more so than solving Problem (b).\ According to the precepts of \cite{MT3}, in order to compute the \textit{genus $2$}
partition function of $V$ it is, however, necessary to first solve Problem (c).

\medskip
An example that influences the present work is that of the Heisenberg VOA $M$.\ In this case the answer to Problem (a) is negative, but both Problems (b) and (c) have been solved.\
In fact Problem (b) was first answered in \cite{DMN} and the result may be stated as follows:\ every trace function satisfies
\begin{eqnarray}\label{Mtrace}
Z_M(u, \tau) \in \mathbf{C}[E_2, E_4, E_6]Z_M(\mathbf{1}, \tau),
\end{eqnarray}
where $\mathbf{C}[E_2, E_4, E_6]$ is the algebra of quasimodular forms and 
$$Z_M(\mathbf{1}, \tau)= \sum_n \dim  M_n q^{n-c/24}$$ 
is the 
 \textit{graded character} of $M$.\ Moreover every such function occurs as a $1$-point function for some $u$.\ An analogous, but somewhat more complicated  result
 also holds for lattice theories $V_L$ (loc.\ cit.)

\medskip
Problem (c) was resolved for $M$ and lattice VOAs $V_L$ in \cite{MT2}.\ The general approach, which  goes back to \cite{Z}, was also used in \cite{DMN} and streamlined in \cite{MT1,MT2}.\ The 
basic idea is to establish a certain type of recursion formula for $1$-point functions, called \textit{Zhu recursion},  that  describes $Z_V(u, \tau)$ in terms of modular data and trace functions of states of lower weight than $u$.\ The recursion formula is a general phenomenon that holds for \textit{any} VOA.\ Difficulty can arise in trying to solve the recursion, and this depends to a large extent on the nature of $V$.\ This process is important for the methodology of the present paper, and we provide some background  in Sections \ref{cylinder}-\ref{SZR1}
closely following  \cite{MT1}.\  In order to give an impression of the shape these explicit results are given in, we cite the following special case of the results in \cite{MT1,MT2} for the Heisenberg algebra of rank $k=1$ (for details on the notation, we refer to Sections~\ref{SS2.3} and~\ref{cylinder}).
\begin{thm}\label{thmMT} Let $M$ be the rank $k=1$ Heisenberg VOA and choose  $h\in H$ with $(h, h)=1$.\ Let $u=h[-n_1]h[-n_2]...h[-n_p]\mathbf{1},\ \ (n_i\geq 1)$ be a state in $M$.\ Then we have that
\begin{eqnarray}\label{sqM}
Z_M(u, \tau)= \left(\sum_{\sigma} \prod_{(rs)} \hat{E}_{n_r+n_s}(\tau)\right) Z_M(\tau).
\end{eqnarray}
Here, $\sigma$ ranges over fixed-point-free involutions in the symmetric group on the index set  $\{1, ..., p\}$ and
$(rs)$ ranges over the transpositions of $S_p$ whose product is equal to $\sigma$.\ $\hfill\Box$
\end{thm}
The square bracket notation $h[n]$ implicitly refers to the 'good' basis of $M$ and will be explained in detail in Section~\ref{cylinder}.\ Note that if we replace $u$ by a word in round bracket modes $h(n)$, for example as in (\ref{evenh}), then the resulting trace function becomes unpalatable - a sum of expressions like (\ref{sqM}) with varying $p$.\
Since $Z_M(\tau)$ is modular, albeit of weight $-\tfrac{1}{2}$ (cf. (\ref{ptnfuncs})) we see from (\ref{sqM})  that the source of the failure of modularity is the possible occurrence of the only quasimodular Eisenstein series
$E_2(\tau)$ (see Section~\ref{SS2.1}).

\medskip
Our main results solve Problem (c) for the symmetrized VOAs
$M^+$ and $V_L^+$, i.e., we give explicit formulas for the $1$-point functions associated to a certain basis for each of these VOAs.\ It transpires that these $\ZZ_2$-fixed point VOAs are not well-suited to Zhu recursion and our principal technique hinges on adapting the Zhu recursion to a $\ZZ_2$-orbifold setting.\ To explain this it will help to consider a
general VOA $V$ admitting an involutive automorphism $t$. Thus there is a decomposition 
\begin{eqnarray*}
V= V^+\oplus V^{-}
\end{eqnarray*}
where the summands are the $\pm 1$-eigenspaces of $t$, and of course $V^+$ is a subVOA of $V$.\ Now recall the definition of $2$-point functions for $V$:
\begin{eqnarray}\label{new2point}
&&F_V((u, z_1), (v, z_2), \tau):= \Tr_V Y_V(q_1^{L(0)}u, q_1)Y_V(q_2^{L(0)}v, q_2)q^{L(0)-c/24}\quad (u, v\in V)
\end{eqnarray}
where we set $q_i:=e^{z_i}$, which is in fact a sort of Jacobi form in the variables $\tau$ and $z:=z_1-z_2$.\ In what is perhaps the crucial part of Zhu recursion, one seeks a formula describing $F_V$
in terms of modular data and $1$-point functions for $V$.\ This step was isolated in \cite{MT1, MT2} and called \textit{Zhu reduction}.\ One has the usual Zhu reduction formulas 
which may be applied both to $F_V$ and $F_{V^+}$ (cf. Section \ref{SZR}) but this is not quite enough for our purposes.\ We observed that expressions of the following ilk:
\begin{eqnarray}\label{verynew2point}
&&F_{V^+}((u, z_1), (v, z_2), \tau)\quad   (u, v\in V^{-})
\end{eqnarray}
make perfectly good sense. It amounts to an expansion of the domain of definition of the $2$-point function for $V^+$.\ Better than that, the method of proof of the original Zhu reduction
formula also applies (with some adjustments) with(\ref{verynew2point}) in place of (\ref{new2point}).\ This is what we call $\ZZ_2$-twisted Zhu reduction.\ Because of the added complexity, the 
$\ZZ_2$-twisted reduction formula (Theorem \ref{thmZ22} in
 Subsection \ref{SSZRreduc}), is correspondingly more complicated than its untwisted precursor.\ And similarly there is a general $\ZZ_2$-twisted recursion, which is Theorem \ref{thmZ2Z}.\
 Solving the twisted recursion can be awkward, but it is manageable for $M^+$ and $V_L^+$.

\medskip
Consider the case when $V=M$.\
 The $1$-point functions for $M^+$
are described in Theorem \ref{thmM+}.\ The general answer is a more complicated version of Theorem \ref{thmMT} and we will not reproduce the exact result here.\ Suffice it to say that every $1$-point function
$Z_{M^+}(u, \tau)$ is a linear combination of $Z_{M^+}(\mathbf{1}, \tau)$ and $Z_{M}(\mathbf{1}, \tau)$ with coefficients which are quasimodular forms on $\Gamma_0(2)$.\
Expressions for the VOA characters  $Z_{M^+}(\mathbf{1}, \tau)$ and $Z_{M}(\mathbf{1}, \tau)$ are well-known (cf.\ Lemma \ref{lemchars}).\ They are both expressible in terms of the $\eta$-function,
but $Z_{M^+}(\mathbf{1}, \tau)$ is itself \textit{not} a modular form of any single weight but rather one of mixed weight, i.e., a sum of modular forms of different weights, $0$ and $-k/2$ in this case.

\medskip
The explicit $1$-point functions for $V_L^+$ are yet more complicated than those for $M^+$, even though they are true modular forms (on a congruence subgroup).\ The
formulas in question  are established in Theorems \ref{thmlattrecur} and \ref{thmVL+trace} together with Remark \ref{rmkfalphatrace}.

\medskip
By the aforementioned works of Zhu \cite{Z}, Huang \cite{H}, Dong-Ng \cite{DN}, and Dong-Li-Mason \cite{DLM} the 1-point functions $Z_{V_L^+}(u,\tau)$ for a homogeneous state $u$ of weight $k$ in the square-bracket grading (again see Section~\ref{cylinder}) is a modular form of the same weight $k$ for some congruence subgroup of $\SL_2(\ZZ)$. We can indeed deduce this fact independently from our explicit results in Theorems~\ref{thmlattrecur} and \ref{thmVL+trace} as well.\ For those states covered by Theorem~\ref{thmlattrecur}, the modularity is immediate from the formula together with Remark~\ref{rmkfalphatrace}. For those considered in Theorem~\ref{thmVL+trace}, modularity is indeed not quite obvious at first glance, as our formula involves non-trivial combinations of quasimodular forms which define modular forms in a non-trivial way. We prove this in Proposition~\ref{propmod} using the theory of Jacobi-like forms (cf.\  Section~\ref{secjacobi}).

\medskip

The rest of the paper is organised as follows: In Section~\ref{S2} we recall some necessary background material on modular forms and elliptic functions (Section~\ref{SS2.1}), Jacobi-like forms (Section~\ref{secjacobi}), and VOAs (Sections~\ref{SS2.2} to \ref{SS2.3}). In Section~\ref{cylinder} we introduce the concept of VOAs \emph{on a cylinder}, which essentially yields a different grading of a VOA $V$ which turns out to be very convenient in our context. Section~\ref{SZR} recalls standard Zhu recursion theory, a new $\ZZ_2$-twisted version of which we then prove in Section~\ref{SZRnew}. Finally, Section~\ref{sec1point} contains our main results, which are explicit formulas for the $1$-point functions in the symmetrized Heisenberg algebra (Section~\ref{sec1pointM}) and in symmetrized lattice VOAs (Section~\ref{secsymlat}).

\section{Notation and Background}\label{S2}
The statement of our main results use a fair amount of notation. Some of it, especially that pertaining to certain modular forms and elliptic-type functions,
is nonstandard and tailor-made for our situation.\ All of this notation is explained here along with some background.

\subsection{Modular forms and elliptic-type functions}\label{SS2.1} Let $\mathbf{H}$ denote the complex upper-half plane, 
$\tau$ a typical element in $\mathbf{H}$ and $q:=e^{2\pi i \tau}$.\ 
The Dedekind eta-function is defined as
\begin{eqnarray*}
\eta(\tau)=q^{1/24}\prod_{n\geq 1} (1-q^n).
\end{eqnarray*}
The Bernoulli numbers $B_k$ are defined by 
\begin{eqnarray*}
\sum_{k\geq 0} \frac{B_k}{k!}z^k:=\frac{z}{(e^z-1)}.
\end{eqnarray*}
The Eisenstein series for positive integers $k$ are defined by
\begin{eqnarray*}
E_k(\tau):= \left\{ \begin{array}{cc}
-\frac{B_k}{k!}+\frac{2}{(k-1)!}\sum_{n\geq 1} \sigma_{k-1}(n)q^n,\ & k\ \mbox{even}\\
0, & k\ \mbox{odd}
\end{array}  \right.
\end{eqnarray*}
with $\sigma_{k-1}(m):=\sum_{d\mid m} d^{k-1}$ denotes the usual divisor sum function.\ We also need certain level $2$ Eisenstein series defined for positive $k$ by
\begin{eqnarray*}
&&F_k(\tau):=2E_k(2\tau)-E_k(\tau).
\end{eqnarray*}

Note that $E_k(\tau)$ is a \emph{modular form} of weight $k$ for the full modular group $\mathrm{SL}(\ZZ)$ as long as $k\neq2$.\ Recall that a modular form of some (integer) weight $k$ for a finite index subgroup $\Gamma\leq\mathrm{SL}_2(\ZZ)$ is a holomorphic function on $\mathbf H$ that satisfies the transformation law
$$f\left(\frac{a\tau+b}{c\tau+d}\right) = (c\tau+d)^kf(\tau)$$
for all $\left(\begin{smallmatrix} a & b \\ c & d\end{smallmatrix}\right)\in\Gamma$ and $\tau\in\mathbf H$, as well as certain growth conditions. For background on modular forms we refer the reader e.g. to \cite{123} or one of the many textbooks on the subject. 
For $k=2$, $E_2$ is not actually modular, but rather a \emph{quasimodular} form of weight $2$. We refrain from giving the precise definition of this term (for this see for instance \cite[Section 5.3]{123}), but rather note that $E_2(\tau)+\frac{1}{4\pi\Im(\tau)}$ transforms like a modular form of weight $2$ for $\SL_2(\ZZ)$ 
(without being holomorphic of course).
Note that the functions $F_k$ are modular forms of weight $k$ for $\Gamma_0(2):=\left\{\left(\begin{smallmatrix} a & b \\ c & d\end{smallmatrix}\right) \: :\: c\equiv 0\pmod 2\right\}$ for every $k$, including $k=2$. 
 
 \medskip
For each pair $(m, n)$ of positive integers we introduce renormalized Eisenstein series 
\begin{eqnarray}
&&\widehat{E}_{m+n}(\tau):= (-1)^{n+1}n{m+n-1\choose n}E_{m+n}(\tau), \label{eqhatE}\\
&&\widehat{F}_{m+n}(\tau):= (-1)^{n+1}n{m+n-1\choose n}F_{m+n}(\tau)\label{eqhatF}.
\end{eqnarray}
Observe that these functions are symmetric in $m$ and $n$.

\begin{rmk} These particular quasimodular forms are ubiquitous throughout the present paper, and that makes the notation very convenient, however a little care is needed
when using it.\ Thus each one of $\widehat{E}_{1+5}(\tau), \widehat{E}_{2+4}(\tau), \widehat{E}_{3+3}(\tau)$ is an integral multiple of $E_6(\tau)$, but the particular multiples are 
\textit{not} the same. 
\end{rmk}

\medskip
We also make use of certain elliptic-type functions as in \cite{MT1,MT2}.\ 
We shall need these functions as both Lambert series and Laurent series.\ The most basic of them is defined as a Lambert series
\begin{eqnarray}\label{P1def}
&&\ \ \  P_1(z, \tau):= \tfrac{1}{2}+\sum_{0\neq n\in\ZZ} \frac{e^{nz}}{(1-q^n)}=-\tfrac{1}{z}+\sum_{k\geq 1} E_{2k}(\tau)z^{2k-1}=E_2(\tau)z-\zeta(z, \tau)
\end{eqnarray}
where $\zeta(z, \tau)$ is the Weierstrass zeta function for the lattice $\Lambda=2\pi i(\ZZ\oplus \ZZ\tau)$, i.e.
\begin{eqnarray}\label{eqzetadef}
\zeta(z,\tau)=\frac 1z+\sum_{\omega\in \Lambda}\left(\frac{1}{z-\omega}+\frac{1}{\omega}+\frac{z}{\omega^2}\right),\quad z\in\CC\setminus\Lambda.
\end{eqnarray}

 In \cite{MT1} this function was introduced and used with a constant $-\tfrac{1}{2}$ in the Lambert series, although this typo makes no difference to the calculations in \textit{loc.\ cit.} since only the $m^{th}$ derivatives with $m\geq1$ (with respect to $z$) intervene in the proofs there.\ But this no longer pertains in 
our present context and it is important that the correct  constant $+\tfrac{1}{2}$ is in place.\  For the reader's convenience the expression as a Laurent series is established in Lemma \ref{lemPQ1} below.\ See also \cite[Eq.(3.11)]{Z}.\ The derivatives in question are 
\begin{eqnarray}\label{eqP1m}
P_1^{(m)}(z, \tau)=  \frac{(-1)^{m+1}}{z^{m+1}} +m!\sum_{k\geq 1} {2k-1\choose m} E_{2k}(\tau)z^{2k-m-1} \ \ \  (m\geq 0).
\end{eqnarray}

We shall also need a related function
\begin{eqnarray*}
Q_1(z, \tau):= \sum_{n\in\ZZ} \tfrac{e^{nz}}{1+q^n},
\end{eqnarray*}
whose higher $z$-derivatives have series expansions as follows (cf.\ Lemma \ref{lemPQ1}):
\begin{eqnarray}\label{Q1def}
&&\ \ \ \ Q_1^{(m)}(z, \tau):=\sum_{n\in\ZZ}\frac{n^me^{nz}}{1+q^n}=\frac{(-1)^{m+1}}{z^{m+1}}      +m!\sum_{k\geq 1} {2k-1\choose m} F_{2k}(\tau)z^{2k-m-1} \ \ (m\geq0).
\end{eqnarray}

\begin{lem}\label{lemPQ1}The following hold:\\
(a) Laurent series expansions for $P_1(z, \tau)$ and $Q_1(z, \tau)$ are as in (\ref{P1def}) and (\ref{Q1def}) respectively.\\
(b)\ $P_1^{(m)}(z, \tau)$ and $Q_1^{(m)}(z, \tau)$ for $m\geq0$ are odd functions if $m$ is even and even functions if $m$ is odd.\\
(c) We have $Q_1(z, \tau) =2P_1(z, 2\tau) - P_1(z, \tau)$ and $Q_1(z,\tau)$ is an elliptic function for the lattice $2\pi i(\ZZ\oplus 2\tau\ZZ)$.
\end{lem}
\begin{proof}  To prove all parts of the Lemma it suffices to establish the   series expansion for $P_1(z, \tau)$ and prove part (c).\ We begin with the latter.\
We have
\begin{eqnarray*}
&&Q_1(z, \tau)= \frac{1}{2}+ \sum_{0\neq n\in\ZZ}\frac{e^{nz}}{1-q^{2n}}-\sum_{0\neq n\in\ZZ}\frac{e^{nz}q^n}{1-q^{2n}}\\
&&=P_1(z, 2\tau)  -  \frac{1}{2}    \left\{ \sum_{0\neq n\in\ZZ}\frac{e^{nz}}{1-q^{n}} -   \sum_{0\neq n\in\ZZ}\frac{e^{nz}}{1+q^{n}}      \right\}        \\
&&=  P_1(z, 2\tau)  -  \frac{1}{2} \left\{ \left(\frac{1}{2}+ \sum_{0\neq n\in\ZZ}\frac{e^{nz}}{1-q^{n}}\right) -   \sum_{n\in\ZZ}\frac{e^{nz}}{1+q^{n}}      \right\}   \\
&&= P_1(z, 2\tau)-\frac{1}{2}P_1(z, \tau) + \frac{1}{2}Q_1(z, \tau).\ 
\end{eqnarray*}
The first claim of part (c) follows immediately.\ For the second claim we note that it is a classical result dating back to Eisenstein that the function
$$\zeta^*(z,\tau)=\zeta(z,\tau)-zE_2(\tau) -\frac{1}{2\pi \Im(\tau)}\Re(z)=P_1(z,\tau)-\frac{1}{2\pi \Im(\tau)}\Re(z)$$
transforms like an elliptic function for the lattice $\Lambda$ (see for instance \cite{Rolen, Weil}), except that it is no longer meromorphic in $z$.  It follows therefore that 
$$Q_1(z,\tau)=2P_1(z,2\tau)-P_1(z,\tau)=2\zeta^*(z,2\tau)-\zeta^*(z,\tau)$$
is indeed elliptic for the sublattice $2\pi i(\ZZ\oplus 2\tau\ZZ)$ of $\Lambda$, since the non-holomorphic terms in $\zeta^*$ cancel.

\medskip
We now show the representation for $P_1$ given in \eqref{P1def}. For this we begin by noting the following identity, which is a direct consequence of the well-known partial fraction decomposition of the cotangent function (see e.g. \cite[Satz III.7.13]{Freitag}),
\begin{eqnarray}\label{eqcot}
-\frac 12\cdot \frac{e^z+1}{1-e^z}=\frac{1}{z}+\sum_{n\geq 1} \left(\frac{1}{z-2\pi in}+\frac{1}{z+2\pi in}\right),\quad z\in\CC\setminus 2\pi i\ZZ,
\end{eqnarray}
where the series on the right-hand side converges absolutely uniformly on any compact subset of $\CC\setminus 2\pi i\ZZ$.  By fixing the order of summation in the definition of the Weierstrass $\zeta$-function, one obtains the following expression for each $z\in\CC\setminus\Lambda$ (see for instance \cite[Eq. (8)]{Zemel}).
\begin{gather}\label{eqzetaconv}
\begin{aligned}
\zeta(z,\tau)-zE_2(\tau)=&\frac 1z+\sum_{n\geq 1} \left(\frac{1}{z-2\pi in}+\frac{1}{z+2\pi in}\right)\\
                       & \qquad+\sum_{m\geq 1}\left[\frac{1}{z+2\pi im\tau}+\sum_{n\geq 1} \left(\frac{1}{z+2\pi im\tau-2\pi in}+\frac{1}{z+2\pi im\tau+2\pi in}\right)\right.\\
                       &\qquad\qquad\left.\frac{1}{z-2\pi im\tau}+\sum_{n\geq 1} \left(\frac{1}{z-2\pi im\tau-2\pi in}+\frac{1}{z-2\pi im\tau+2\pi in}\right)\right],
                       \end{aligned}
\end{gather}
where each of the sums on the right-hand side is absolutely convergent.  Using \eqref{eqcot}, we therefore find that
$$\zeta(z,\tau)-E_2(\tau)z=-\frac 12\left[\frac{e^z+1}{1-e^z}+\sum_{m\geq 1}\left(\frac{e^zq^m+1}{1-e^zq^m}+\frac{e^zq^{-m}+1}{1-e^zq^{-m}}\right)\right].$$
For $-2\pi\Im(\tau)<\Re(z)<0$, we can rewrite this using the usual geometric series as
\begin{align*}
 &-\frac 12\left[\frac{e^z+1}{1-e^z}+2\sum_{m\geq 1}\sum_{n\geq 1}(e^{nz}-e^{-nz})q^{mn}\right]\\
=&-\frac 12\left[\frac{e^z+1}{1-e^z}+2\sum_{n\geq 1}(e^{nz}-e^{-nz})\frac{q^n}{1-q^n}\right]\\
=&-\frac 12\left[1+2\sum_{n\geq 1}e^{nz}+2\sum_{n\geq 1}(e^{nz}-e^{-nz})\frac{q^n}{1-q^n}\right]\\
=&-\sum_{n\geq 1}\left(\frac{e^{nz}}{1-q^n}-\frac{e^{-nz}q^n}{1-q^n}\right)-\frac 12\\
=&-\sum_{n\neq 0} \frac{e^{nz}}{1-q^n}-\frac 12\\
=&-P_1(z,\tau),
\end{align*}
showing the last equation in \eqref{P1def}. The first one follows immediately from the well-known Laurent expansion of $\zeta(z,\tau)$ (see for instance \cite[Chapter III, \S 7]{Weil}).
\end{proof}

\subsection{Theta functions and Jacobi-like forms}\label{secjacobi}
The theta function of an even lattice $L$ of rank $k$ with a positive definite bilinear form $(\ ,\ )$ is defined by
$$\theta_L(\tau)=\sum_{\alpha\in L}q^{(\alpha,\alpha)/2}.$$
Historically, such functions are the first examples of modular forms,  going back at least to the works of Euler and they are still a vital tool in the study of quadratic forms and lattices.  In Section~\ref{secsymlat} we shall encounter the following variant of the theta function of $L$,
\begin{gather}\label{eqthetavm}
\theta_L(\tau,v,m)=\sum_{\alpha\in L} (v,\alpha)^mq^{(\alpha,\alpha)/2},
\end{gather}
for any vector $v\in\CC\otimes L$ and nonnegative integer $m\geq 0$.  Notice that $\theta_L(\tau,v,m)$ is identically $0$ whenever $m$ is odd and that $\theta_L(\tau,v,0)=\theta_L(\tau)$. A famous result by Hecke and Schoeneberg \cite{Hecke,Schoeneberg} states that if $(v,v)=0$ these functions $\theta_L(\tau,v,m)$ are indeed holomorphic modular forms of weight $k/2+m$ with respect to the group $\Gamma_0(N)$, where $N$ is the smallest positive integer such that $NG(L)^{-1}$ with $G(L)$ denoting the Gram matrix of $L$, has integral entries which are even on the diagonal, called the \emph{level} of $L$, with some (explict) multiplier system depending on the lattice. In fact, for $m>0$, $\theta_L(\tau,v,m)$ is even a cusp form.

\medskip
There is an analogous statement for non-isotropic vectors $v$, i.e. where $(v,v)\neq 0$. It is clear from the definition that one can then renormalize $v$ such that $(v,v)=1$ without affecting the modularity properties. The transformation law is then most easily formulated in terms of so-called \emph{Jacobi-like forms}.

\medskip
These objects
have been introduced by Zagier \cite{Zagier}. By definition these are holomorphic functions
\begin{gather}\label{eqphi}
\phi(\tau,X)=\sum_{n=0}^\infty \phi_n(\tau)(2\pi iX)^n
\end{gather}
on $\mathbf H\times\CC$ satisfying the transformation law
$$\phi\left(\frac{a\tau+b}{c\tau+d},\frac{X}{(c\tau+d)^2}\right)=(c\tau+d)^k\exp\left(\frac{mcX}{c\tau+d}\right)\phi(\tau,X), \quad \begin{pmatrix}
a & b \\ c & d
\end{pmatrix}\in\Gamma\leq\SL_2(\ZZ),$$
for some integer $k$, called the \emph{weight}, and some complex number $m\in\CC$ called the \emph{index}.\ This definition can of course be modified in a straightforward way to allow for multiplier systems.\ We then have the following result (see \cite[Theorem 1]{DMtheta}):

\begin{thm}\label{thmSchoeneberg}
Let $L$ be an even lattice of rank $k$ and $v\in\CC\otimes L$. Then the function 
\begin{gather}\label{eqThetaX}
\Theta_L(\tau,v,X):=\sum_{m\geq 0}\frac{2^m}{(2m)!}\theta_L(\tau,v,2m)(2\pi iX)^m
\end{gather}
is a Jacobi-like form of weight $k/2$ and index $(v,v)$.
\end{thm}
It is easy to deduce from the definition that a function $\phi(\tau,X)$ as in \eqref{eqphi} is a Jacobi-like form of weight $k$ and index $0$ if, and only if, the coefficient functions $\phi_n(\tau)$ are all modular of weight $k+2m$, so the aforementioned result by Hecke and Schoeneberg follows directly from Theorem~\ref{thmSchoeneberg}.

We shall also need the following Jacobi-like form,
\begin{gather}\label{eqEhat}
\widetilde E(\tau,X)=\exp(E_2(\tau)\cdot (-2\pi iX)).
\end{gather}
The following is easily seen from the well-known transformation properties of the Eisenstein series $E_2$ under modular transformations.

\begin{lem}\label{lemEhat}
The function $\widetilde E(\tau,X)$ is a Jacobi-like form of weight $0$ and index $1$ for the full modular group.
\end{lem}

\subsection{General definitions for VOAs}\label{SS2.2}
For the convenience of the reader, we recall here and in the following Subsection some basic facts and definitions from the theory of VOAs as far as they are relevant for our purposes. There are several introductory texts on the subject that the interested reader may find beneficial, such as \cite{FLM,LL,MT1}  among many others.

\medskip
A \emph{vertex operator algebra} (VOA) over the complex numbers has  four pieces of data:
\begin{itemize}
\item a $\ZZ$-graded $\CC$-vector space $V=\bigoplus_{k\in\ZZ} V_k$, called the \emph{Fock space}, where $\dim V_k<\infty$ and $V_k=\{0\}$ for $k$ sufficiently small. The elements in $V$ are called \emph{states}.
\item a \emph{state-field correspondence}, which is a $\CC$-linear map
$$Y: V\to \End(V)\llbracket z^{\pm 1}\rrbracket,\quad v\mapsto Y(v,z):=\sum_{n\in\ZZ} v(n)z^{-n-1}$$
associating to each state $v$ a collection of endomorphisms $v(n):V\to V$, $n\in\ZZ$, called the \textit{modes} of $v$.\ The formal distribution $Y(v,z)$ is called a \emph{field}  or \emph{vertex operator}.\ For any pair of states $u, v$, modes are required to satisfy the truncation condition $u(n)v=0$ for $n\gg0$. 
\item a distinguished state $\mathbf{1}\in V$, called the \emph{vacuum}, with the property that for any $v\in V$ we have 
$$Y(v,z)\mathbf 1=v+O(z) \quad\text{(creativity)}.$$
\item a distinguished state $\omega\in V$,  the \emph{Virasoro vector}, such that 
$$Y(\omega,z)=\sum_{n\in\ZZ} \omega(n)z^{-n-1}=\sum_{n\in\ZZ} L(n)z^{-n-2}$$
and the modes $L(n)$ are required to close on a \emph{Virasoro algebra}, i.e., we have the commutator relations
\begin{gather}\label{eqvirasoro}
[L(m),L(n)]=(m-n)L(m+n)+\delta_{m+n,0}(m^3-m)\frac{c}{12}\kappa
\end{gather}
for a central element $\kappa$ and a scalar $c\in\CC$ called the \emph{central charge} of $V$.
\end{itemize}
The fields $Y(u,z),Y(v,z)$ for states $u,v\in V$ are required to be \emph{mutually local}, i.e. there exists an integer $K>0$ such that
\begin{gather}\label{eqlocality}
(z_1-z_2)^K[Y(u,z_1),Y(v,z_2)]=0.
\end{gather}
In addition, the Virasoro operator $L(0)$ must respect the grading,  i.e. if $v\in V_k$ ---we say that $v$ is homogeneous of \emph{(conformal) weight} $k$--- then we must have $L(0)v=kv.$
It is not hard to show that the vacuum vector $\mathbf 1$ is homogeneous of weight $0$ and the Virasoro state $\omega$ is homogenous of weight $2$.\ More generally, if $v$ is homogeneous of weight $k$ then the modes $v(n)$ respect the grading in that the restriction
\begin{gather}\label{eqgrading}
v(n): V_m\rightarrow V_{m+k-n-1}
\end{gather}
is well-defined.\ The operator $T:=L(-1)$, often referred to as the \emph{translation operator}, satisfies the commutation relations
\begin{gather}\label{eqtranslation}
[T,Y(v,z)]=\partial_z Y(v,z),\quad T\mathbf 1=0\quad \text{(translation covariance)}.
\end{gather}

One important consequence of the above axioms which is usually given as an additional axiom for a VOA is the so-called \emph{Jacobi } or \textit{Borcherds} identity, the modal form of which is
\begin{multline}\label{eqborcherds}
\sum_{i\geq 0}(-1)^{i+1}\binom{\ell}{i}(u(\ell+m-i)v(n+i)-(-1)^\ell v(\ell+n-i)u(m+i))\\
=\sum_{i\geq 0}\binom{m}i (u(\ell+i)v)(m+n-i),\quad (\ell,m,n\in\ZZ).
\end{multline}

\medskip
Due to the axiomatically required truncation property for positive modes the summations on both sides in \eqref{eqborcherds}, when applied to any state $w\in V$, are in fact finite and thus well-defined. 
We note two important special cases of \eqref{eqborcherds}. For $\ell=0$ we obtain the \emph{commutator formula}
\begin{gather}\label{eqcommutator}
[u(m),v(n)]=\sum_{i\geq 0}\binom mi (u(i)v)(m+n-i),
\end{gather}
and for $m=0$ we obtain the \emph{associator formula}
\begin{gather}\label{eqassociator}
(u(\ell)v)(n)=\sum_{i\geq 0} (-1)^{i+1}\binom\ell i (u(\ell-i)v(n+i)-(-1)^{\ell}v(\ell+n-i)u(i)).
\end{gather}
In the presence of the other axioms,  \eqref{eqcommutator} and \eqref{eqassociator} together are equivalent to \eqref{eqborcherds}.
\subsection{Background, notation and conventions for Heisenberg and lattice VOAs}\label{SS2.3}

\medskip
We describe the VOAs with which we will be mainly concerned with in the present paper.

\medskip
Let $H=\CC^k$ be a $k$-dimen\-sional complex linear space equipped with a nondegenerate, symmetric bilinear form
$( \ \ , \ \ ): H\times H \rightarrow \CC$.\
$H$ determines the Heisenberg VOA of rank (or central charge) $k$. The corresponding Fock space is an infinite-dimensional symmetric algebra $M:= S(H_1\oplus H_2 \hdots)$ where $H_i$ is a copy of $H$ in degree $i$.\ We do not usually
adorn the symbol $M$ with additional notation: the context will always make clear exactly which Heisenberg we are dealing with.\ 

\medskip
The vertex operator for $h\in H$ is denoted
\begin{eqnarray*}
Y_M(h, z)=Y(h, z)=\sum_{n\in\ZZ} h(n)z^{-n-1}.
\end{eqnarray*}
 If $h_1, ..., h_k$ is an orthonormal basis of $H$ with respect to the bilinear form then $M$ has a basis consisting of states 
 \begin{eqnarray}\label{Mstates}
 h_{i_1}(-n_1)... h_{i_p}(-n_p)\mathbf{1}
 \end{eqnarray}
  for all $p\geq 0$ and integers $n_{i_j}\geq1$.\ The conformal weight of such a state is given by $\sum_{j=1}^p n_j$.\ The modes $h_i(n)$ satisfy the canonical commutator relations
\begin{eqnarray}\label{commrelns}
[h_i(m), h_j(n)]=\delta_{i, j}\delta_{m, -n} m Id_M,
\end{eqnarray}
In particular the order in which the modes $h_{i_j}(-n_j)$ appear in (\ref{Mstates}) is immaterial.

\medskip
Let $L$ be an \textit{even lattice} (which will always mean that $L$ is also integral and positive-definite) with bilinear form $(\ \  , \ \ ): L\times L\rightarrow \ZZ$.\ The $\CC$-linear extension defines a bilinear form on $H:=\CC\otimes L$ and the lattice theory VOA has Fock space 
\begin{eqnarray*}
V_L:=M\otimes\CC[L] = \oplus_{\alpha\in L} M\otimes e^{\alpha}
\end{eqnarray*}
where we conflate $e^0$ with the vacuum state $\mathbf{1}$.\ A set of basis vectors in $V_L$ may be taken to be $u\otimes e^{\alpha}$ where $u$ ranges over a basis of $M$ and $\alpha\in L$.\ If $u\in M$ has weight $k$, then $u\otimes e^\alpha$ has weight $k+(\alpha,\alpha)/2$.\ Note that $M=M\otimes\mathbf{1}\subseteq V_L$ is a subVOA.\ The vertex operators for generic states in $V_L$ are awkward to handle and one of the points of the present paper is that we will not need them to calculate the trace functions.\ 

\medskip
We frequently identify $\mathbf{1}\otimes e^{\alpha}$ with $e^{\alpha}$, so that also $h_{i_1}(-n_1)... h_{i_p}(-n_p)\mathbf{1}\otimes e^{\alpha}=h_{i_1}(-n_1)... h_{i_p}(-n_p)e^{\alpha}$.\
The states $e^{\alpha}$ are singular vectors for $M$ in the sense that we have 
\begin{eqnarray}\label{eqhnealpha}
h(n)e^{\alpha} =  \delta_{n, 0} (h, \alpha) e^{\alpha}\ \ (n\geq 0).
\end{eqnarray}
The exceptional  behaviour of the zero mode $h(0)$ exhibited here is responsible for much of the complication in our trace formulas.

\medskip
Let $t:\alpha\mapsto -\alpha$ be the negating automorphism of $L$.\ There is a canonical lift of $t$ to an  involutive automorphism of $V_L$  that we continue to denote by $t$.\ Its action is defined by
\begin{eqnarray*}
t: h_{i_1}(-n_1)... h_{i_p}(-n_p)\mathbf{1}\otimes e^{\alpha}\mapsto (-1)^ph_{i_1}(-n_1)... h_{i_p}(-n_p)\mathbf{1}\otimes e^{-\alpha}.
\end{eqnarray*}
Note that $t$ leaves invariant the Heisenberg subVOA $M\subseteq V_L$.

\medskip
The two fixed-point subVOAs $V_L^+$ and $M^+$ of $V_L$ consist of the states in $V_L$ and $M$ respectively that are fixed by $t$.\  The complements, i.e., the $-1$ eigenspaces of $t$ in $M$ (resp. $V_L$) are denoted by $M^-$ (resp. $V_L^-$), so that  $M=M^+\oplus M^-$ and $V_L=V_L^+\oplus V_L^-$. 

\medskip
Evidently $M^+$ has basis consisting of the states 
\begin{eqnarray}\label{evenh}
h_{i_1}(-n_1)... h_{i_p}(-n_p)\mathbf{1}\ \ (p\ \mbox{even}).
\end{eqnarray}
As for the states in $V_L^+$, it is convenient to introduce the notation
\begin{eqnarray}\label{fgdef}
f_{\alpha}:= e^{\alpha}+e^{-\alpha}\ ;\ \ \ g_{\alpha}:= e^{\alpha}-e^{-\alpha}\ \ (0\neq \alpha\in L).
\end{eqnarray}
Then $V_L^+$ is spanned by $M^+$ together with the states
\begin{eqnarray}\label{genstates}
&&h_{i_1}(-n_1)... h_{i_p}(-n_p)\mathbf{1}\otimes f_{\alpha} \ \ \  (p\ \mbox{even})\ ;\ \ \ h_{i_1}(-n_1)... h_{i_p}(-n_p)\mathbf{1}\otimes g_{\alpha} \ \ \ (p\ \mbox{odd}).
\end{eqnarray}
There are bases with analogous shapes for the complements $M^-$ and $V_L^-$ as well.

\medskip
We now discuss $1$-point functions.\ For additional background on this topic, see \cite{MT1,MT2,Z}.\ It is convenient to work in a quite general VOA $V$
of central charge $c$ and equipped with its conformal grading
$V = \bigoplus_{n} V_n$.\ That is, $V_n$ is the $L(0)$-eigenspace for eigenvalue $n$.

\medskip
Let  $v\in V_k$.\ The \textit{zero mode}\footnote{This is something of a misnomer} of $v$ is denoted $o(v)$ and defined by
$o(v):=v(k-1)$.\ Each state  $u\in V$ is a finite sum of  homogeneous states, say $u=\sum_k v_k\ (v_k\in V_k)$, and we can extend the zero mode notation to
a map on $V$, defining $o(u):=\sum_k o(v_k)$.\ Each such zero mode leaves every homogeneous space $V_k$ invariant , so we can define a function called a $1$-point function, by the prescription
\begin{eqnarray*}
Z_V(u, \tau)=Z(u, \tau):= \Tr_V o(u)q^{L(0)-c/24} = q^{-c/24}\sum_{n} \Tr_{V_n}o(u) q^n.
\end{eqnarray*}
The \textit{partition function}, or \textit{character} of $V$ is then
\begin{eqnarray*}
Z_V(\tau):= Z_V(\mathbf{1}, \tau)=q^{-c/24}\sum_{n} \dim V_k q^n.
\end{eqnarray*}

The following particular characters occur frequently in our  calculations.
\begin{lem}\label{lemchars} Let  $L$ be an even lattice of rank $k$ with $H = \CC^k$ as before.\ Then
\begin{eqnarray*}\label{ptnfuncs}
&&Z_M(\tau)=\frac{1}{\eta(\tau)^k},\ \ \ \ \ Z_{M^+}(\tau)=\frac{1}{2}\left( \frac{1}{\eta(\tau)^k}+\frac{\eta(\tau)^k}{\eta(2\tau)^k} \right), \\
&&Z_{V_L}(\tau) = \frac{\theta_L(\tau)}{\eta(\tau)^k},\ \ \ \  Z_{V_L+}(\tau)= \frac 12 \left( \frac{\theta_L(\tau)}{\eta(\tau)^k}+  \frac{\eta(\tau)^k}{\eta(2\tau)^k}\right).
\end{eqnarray*}
Here $\theta_L(\tau):=\sum_{\alpha\in L} q^{(\alpha, \alpha)/2}$ is the theta-function of $L$.
\end{lem}
\begin{proof} The characters of $M$ and $V_L$ are well-known and we omit further discussion of these cases.\ We give a short derivation for the other two characters.\ We begin with $M^+$.\ By construction, the dimension of the weight $r$ space in $M^+$ is given by the number of $k$-coloured partitions of $r$ into an even number of parts.\  Just as in the usual proof that the generating function for the usual partition function is given by $q^{1/24}/\eta(\tau)$, one sees that the coefficient of $t^mq^r$ in the expression
$$\prod_{n\geq 1}(1-tq^n)^{-k}$$
yields precisely the number of $k$-coloured partitions of $r$ into exactly $m$ parts, wherefore it follows directly that the number of partitions of $r$ into an even (resp.\ odd) number of parts is exactly the coefficient of $q^r$ in
\begin{gather}\label{eqcolouredpartitions}
\frac{1}{2}\left(\prod_{n\geq 1}(1-q^n)^{-k}\pm \prod_{n\geq 1}(1+q^n)^{-k}\right)=\frac{1}{2}\left(\prod_{n\geq 1}(1-q^n)^{-k}\pm \prod_{n\geq 1}\left(\frac{1-q^n}{1-q^{2n}}\right)^{k}\right).
\end{gather}
This readily yields the claimed formula for the character of $Z_{M^+}$. Using the negative sign in \eqref{eqcolouredpartitions} of course yields the analogous partition function for the $M^+$-module $M^-$.

\medskip
For $V_L^+$ we notice that there is a basis as given in \eqref{genstates}.\ So we find that the partition function of $V_L^+$ is given by
\begin{multline*}
Z_{V_L^+}=Z_{M^+}+Z_{M^+}\sum_{\alpha\in L\setminus\{0\}/\pm } q^{(\alpha,\alpha)/2}+Z_{M^-}\sum_{\alpha\in L\setminus\{0\}/\pm} q^{(\alpha,\alpha)/2}\\
 =Z_{M^+}+Z_{M^+}\frac 12 (\theta_L-1)+Z_{M^-}\frac 12 (\theta_L-1).
\end{multline*}
Using \eqref{eqcolouredpartitions}, this simplifies to
\begin{gather}\label{eqpartitionVLplus}
Z_{V_L^+}(\tau)=\frac 12 \cdot \frac{\theta_L(\tau)}{\eta(\tau)^k}+\frac 12 \cdot \frac{\eta(\tau)^k}{\eta(2\tau)^k},
\end{gather}
as claimed. Note that the formula for $V_L^+$ (stated for the special case where $L=\Lambda$ is the Leech lattice) may be found e.g. in \cite[Remark 10.5.4]{FLM}.
\end{proof}

In general it is not just $1$-point functions that are of interest, but rather $N$-point functions for arbitrary $N$. We won't define them in full generality, but in order to prove Zhu recursion we  require $2$-point functions. These were defined in (\ref{verynew2point}).

\section{VOAs on a cylinder and the square bracket formalism}\label{cylinder}
Suppose that $V =\bigoplus_k V_k$ is a VOA of central charge $c$.\ We are going to define some new endomorphisms $v[n]$ for states $v\in V$, called the
square bracket modes.\ For further details see \cite{MT2,Z}.\ We merely highlight some details that we need.

\medskip
The round and square bracket modes are related as follows:\ for $m\geq 0$ and $v\in V_k$ we have
\begin{eqnarray}
&&Y[v, z] = \sum_{n\in\ZZ} v[n]z^{-n-1}:= Y(e^{zL(0)}v, e^z-1),\label{sqbrackdef} \\
&&v[m]=m!\sum_{i\geq m} c(k,i, m)v(i) \ \ (m\geq0) \label{v[m]def} 
\end{eqnarray}
where
\begin{eqnarray}\label{binomid}
{k-1+x\choose i}=: \sum_{m=0}^i c(k, i, m)x^m\ \ (i\geq 0).
\end{eqnarray}
These formulas follow  from the Definition.\ In fact, the series
(\ref{sqbrackdef}) are the vertex operators for a VOA with Fock space $V$ and the same vacuum vector as $V$.\ The Virasoro vector in the new VOA is then given by $\tilde \omega=\omega-\frac{c}{24}\mathbf 1$, so that even though it has the same Fock space, the grading is different. This is the VOA $V$ 
\textit{on the cylinder}.\ (The name derives from the change of variables $z\mapsto e^z-1$ where $z$ is a local variable on the punctured complex sphere.)
We have
\begin{eqnarray*}
&&\sum_{ m\geq 0} \frac{(k+1-x)^m}{m!}v[m]=\sum_{ m\geq 0} (k+1-x)^m\sum_{i\geq m} c(k, i,m)v(i)\\
&&=\sum_{i\geq 0}\left\{\sum_{0\leq m\leq i} (k+1-x)^mc(k, i, m) \right \}v(i) \notag\\
&&=\sum_{i\geq 0}\left\{ {2k-x\choose i} \right \}v(i)\notag\\
\end{eqnarray*}
where we used (\ref{binomid}).\ Now take $x=2k-n$ in the last identity ($n$ any integer) to obtain
\begin{eqnarray}\label{form1}
\sum_{m\geq0} \frac{(n-k+1)^m}{m!} v[m]=\sum_{i\geq 0} {n\choose i}v(i).
\end{eqnarray}

An example that we will need are the square bracket modes in the case that $V=M$.\ It is well-known, and not hard to show,  that in this case the canonical commutator relations
(\ref{commrelns}) are also satisfied if we replace $h_i(n)$ with $h_i[n]$.

\medskip
Finally, we need the following straightforward conjugation formula
\begin{eqnarray}\label{conjform}
e^{xL(0)}Y_M(v, z)e^{-xL(0)} = Y_M(e^{xL(0)}v, e^xz).
\end{eqnarray}

\section{Standard Zhu theory}\label{SZR}
For the benefit of the reader as well as to highlight the key differences with our new, twisted version of Zhu reduction, in this Section we give a detailed proof of  the original reduction and recursion  formulas following  \cite{MT1}.

\subsection{Standard Zhu reduction}  A reduction theorem expresses $2$-point functions in terms of modular data and $1$-point functions.\ 
Throughout this and the following Section let $V$ be any VOA with central charge $c$ and let $u,v$ be states in $V$. 
\begin{thm}\label{thm1}(Zhu reduction) We have
\begin{eqnarray*}
F_V((u, z_1), (v, z_2), \tau)= \Tr_V o(u)o(v)q^{L(0)-c/24}+\sum_{m\geq 1} \frac{(-1)^{m+1}}{m!}P_1^{(m)}(z_{12}, \tau)Z_V(u[m]v, \tau),
\end{eqnarray*}
where $z_{12}:=z_1-z_2=-z_{21}$.
\end{thm}
\begin{proof} Let $u\in V_k$ and recall that $q_i:=e^{z_i}\ (i=1, 2)$.\ Then from (\ref{verynew2point}),
\begin{eqnarray*}
F_V((u, z_1), (v, z_2), \tau):= \sum_{n\in\ZZ} q_1^{k-n-1} \Tr_V \left\{u(n)Y_V(q_2^{L(0)}v, q_2)\right\}q^{L(0)-c/24}.
\end{eqnarray*}

The commutator formula \eqref{eqcommutator} may alternately be stated as
\begin{eqnarray*}
&&[u(r), Y_V(v, z)] = \sum_{i\geq 0} {r\choose i} Y_V(u(i)v, z)z^{r-i}.
\end{eqnarray*}
Thus we see that
\begin{eqnarray*}
&&[u(n), Y_V(q_2^{L(0)}v, q_2)]= \sum_{i\geq 0} {n\choose i} Y_V(u(i)q_2^{L(0)}v, q_2)q_2^{n-i}\\
&&=q_2^{n-k+1}Y_V\left(q_2^{L(0)}\sum_{i\geq0} {n\choose i}u(i)v, q_2\right)\\
&&=q_2^{n-k+1}Y_V\left(q_2^{L(0)}\sum_{i\geq0} {n\choose i}u(i)v, q_2\right)\\
&&=q_2^{n-k+1}Y_V\left(q_2^{L(0)}\sum_{m\geq0} \frac{(n-k+1)^m}{m!}u[m]v, q_2\right)\\
&&=q_2^{n-k+1}\sum_{m\geq0} \frac{(n-k+1)^m}{m!}Y_V(q_2^{L(0)}u[m]v, q_2).
\end{eqnarray*}

Now we apply this formula and trace over $V$ to find that if $r:=n-k+1$ then
\begin{eqnarray}\label{formular}
&&\ \ \ \ \ \ \ \ \ \ \ \ \  \Tr_V  \left\{u(n)Y_V(q_2^{L(0)}v, q_2) \right\}q^{L(0)-c/24} \notag\\
&&\ \ \ \ \  = \Tr_V q_2^{r}\sum_{m\geq0} \frac{r^m}{m!}Y_V(q_2^{L(0)}u[m]v, q_2) q^{L(0)-c/24}+
\Tr_V  \left\{Y_V(q_2^{L(0)}v, q_2)u(n) \right\}q^{L(0)-c/24} \notag\\
&&\ \ \ \ \ \ \ \ = q_2^{r}\sum_{m\geq0} \frac{r^m}{m!}Z_V(u[m]v, \tau) +
q^r \Tr_V \left\{Y_V(q_2^{L(0)}v, q_2)q^{L(0)-c/24}u_M(n) \right\}
\end{eqnarray}
where we used (\ref{conjform}) to get the last term.\ Since $\Tr AB= \Tr BA$ we find that
\begin{eqnarray*}
(1-q^r)  \Tr_V  \left\{u_M(n)Y_V(q_2^{L(0)}v, q_2) \right\}q^{L(0)-c/24}= q_2^{r}\sum_{m\geq0} \frac{r^m}{m!}Z_V(u[m]v, \tau).
\end{eqnarray*}
So if $r\neq0$ then
\begin{eqnarray}\label{formuntw}
  \Tr_V \left\{u(n)Y_V(q_2^{L(0)}v, q_2) \right\}q^{L(0)-c/24}= \frac{q_2^{r}}{(1-q^r)}\sum_{m\geq0} \frac{r^m}{m!}Z_V(u[m]v, \tau).
\end{eqnarray}
The case $r=0$ shows that 
\begin{eqnarray}\label{[0]mode}
Z_V(u[0]v, \tau)=0
\end{eqnarray}

\bigskip

Now we return to our $2$-point function to find that
\begin{eqnarray*}
&&F_V((u, z_1), (v, z_2), \tau):= \sum_{n\in\ZZ} q_1^{k-n-1} \Tr_V \left\{u(n)Y_V(q_2^{L(0)}v, q_2)\right\}q^{L(0)-c/24}\\
&&=\sum_{n\in\ZZ,\ r\neq0} q_1^{-r} \frac{q_2^{r}}{(1-q^r)}\sum_{m\geq0} \frac{r^m}{m!}Z_V(u[m]v, \tau)+
 \Tr_V \left\{o(u)Y_V(q_2^{L(0)}v, q_2)\right\}q^{L(0)-c/24}\\
 &&=\Tr_V o(u)o(v)q^{L(0)-c/24}+\sum_{m\geq1} \frac{1}{m!}Z_V(u[m]v, \tau) \sum_{0\neq r\in\ZZ} r^m\frac{q_{12}^{-r}}{(1-q^r)}\\
  &&=\Tr_V o(u)o(v)q^{L(0)-c/24}+\sum_{m\geq1} \frac{1}{m!}Z_V(u[m]v, \tau) \sum_{0\neq r\in\ZZ} r^m\frac{q_{21}^{r}}{(1-q^r)}\\
   &&=\Tr_V o(u)o(v)q^{L(0)-c/24}+\sum_{m\geq1} \frac{1}{m!}Z_V(u[m]v, \tau) P_1^{(m)}(z_{21}, \tau)\\
    &&=\Tr_V o(u)o(v)q^{L(0)-c/24}+\sum_{m\geq1} \frac{(-1)^{m+1}}{m!}Z_V(u[m]v, \tau) P_1^{(m)}(z_{12}, \tau)
\end{eqnarray*}
where we have used (\ref{[0]mode}) as well as Lemma \ref{lemPQ1}(b).\
 This completes the proof of the Theorem.
\end{proof}

\subsection{Standard Zhu recursion}\label{SZR1}
A `recursion' expresses $1$-point functions in terms of modular data and $1$-point functions of states of lower weight.
\begin{thm}\label{thm2}(Zhu recursion) Suppose that $n\geq 1$.\ Then we have
\begin{eqnarray*}
Z_V(u[-n]v, \tau) = \delta_{n, 1}\Tr_V o(u)o(v)q^{L(0)-c/24}+ \sum_{m\geq 1}\frac{1}{m}\widehat{E}_{n+m}(\tau)Z_V(u[m]v, \tau).
\end{eqnarray*}
\end{thm}

To begin the proof of Zhu Recursion we use another reduction Theorem:
\begin{thm}\label{thm3}  We have
\begin{eqnarray*}
&&F_V((u, z_1), (v, z_2), \tau)= Z_V(Y[u, z_{12}]v, \tau)
\end{eqnarray*}
\end{thm}
\begin{proof} We begin by giving the associativity formula \eqref{eqassociator}, which  says that for all large enough integers $k$ we have
\begin{eqnarray*}
(z_1+z_2)^kY_V(u, z_1+z_2)Y_V(v, z_2)=(z_1+z_2)^k Y_V(Y(u, z_1)v, z_2)
\end{eqnarray*}
Upon taking traces the factors $(z_1+z_2)^k$ conveniently disappear, leading to
\begin{eqnarray}\label{assoc1}
\Tr_V Y_V(u, z_1+z_2)Y_V(v, z_2)q^{L(0)-c/24}=\Tr_V Y_V(Y(u, z_1)v, z_2)q^{L(0)-c/24}.
\end{eqnarray}

\medskip
Turning to the $2$-point function (\ref{verynew2point}), we now calculate using (\ref{assoc1}) that
\begin{eqnarray*}
&&F_V((u, z_1), (v, z_2), \tau)=  \Tr_V Y_V(q_1^{L(0)}u, q_1)Y_V(q_2^{L(0)}v, q_2)q^{L(0)-c/24}\\
&&\ \ \ \ \ \ =  \Tr_V Y_V(Y(q_1^{L(0)}u, q_1-q_2)q_2^{L(0)}v, q_2)q^{L(0)-c/24}\\
&&\ \ \ \ \ \ =  \Tr_V Y_V (q_2^{L(0)}Y(q_2^{-L(0)}q_1^{L(0)}u, q_2^{-1}(q_1-q_2))v, q_2)q^{L(0)-c/24}
\end{eqnarray*}
(\mbox{using (\ref{conjform}) with $M=V$})
\begin{eqnarray*}
&&= \Tr_V Y_V(q_2^{L(0)}Y(e^{z_{12}L(0)}u, e^{z_{12}}-1)v, q_2)q^{L(0)-c/24}\\
&&=\Tr_V Y_V(q_2^{L(0)}Y[u, z_{12}]v, q_2)q^{L(0)-c/24}\\
&&=Z_V (Y[u, z_{12}]v,\tau).
\end{eqnarray*}
The Theorem is proved.
\end{proof}

\medskip\noindent
\textit{Proof of Zhu recursion formula}: Combining Theorems \ref{thm1} and \ref{thm3} yields
\begin{eqnarray*}
 Z_V(Y[u, z_{12}]v, \tau)=
\Tr_V o(u)o(v)q^{L(0)-c/24}+\sum_{m\geq 1} \frac{(-1)^{m+1}}{m!}P_1^{(m)}(z_{12}, \tau)Z_V(u[m]v, \tau).
\end{eqnarray*}

Now take coefficients of $z_{12}^{n-1}$ on both sides to obtain
\begin{eqnarray*}
&&Z_V(u[-n]v, \tau) = \delta_{n, 1} \Tr_V o(u)o(v)q^{L(0)-c/24} + \\
&&\ \ \ \ \ \ \  \sum_{m\geq 1} \frac{(-1)^{m+1}}{m!}\left(m!{m+n-1\choose m}E_{m+n}(\tau)  \right)Z_V(u[m]v, \tau),
\end{eqnarray*}
and Theorem \ref{thm2} is now proved.\ $\hfill \Box$

\begin{rmk}
It is straightforward to define $1$- and $2$-point functions for a $V$-module $A$ instead of the VOA $V$ itself. The formulas in Theorems~\ref{thm1}-\ref{thm3} hold in this situation as well (replacing subscript $V$ by $A$), with almost verbatim the same proofs.\ See \cite{FHL} for additional background.\ 
We use Theorem~\ref{thm2} for modules in the proof of Lemma~\ref{lemVL}.
\end{rmk}

\section{$\ZZ_2$-twisted Zhu reduction}\label{SZRnew}  In this Section we prove our version of Zhu reduction formula adapted to involutive automorphisms.

\subsection{Modified $2$-point functions}
For the remainder of the  paper we will be operating in the following context, as discussed in the Introduction:
\begin{eqnarray}\label{setup}
&& V\ \mbox{is a VOA admitting an involutive automorphism $t$,}\\
&& V= V^+\oplus V^{-}\ \mbox{is the decomposition of $V$ into $\pm 1$-eigenspaces for $t$}.\notag
\end{eqnarray}
 $V^+$ is a subVOA of $V$ and $V^{-}$ is a $V^+$-module.

\medskip
Fix a pair of states $u, v\in V^{-}$.\ We repeat the definition  of our modified $2$-point function from the Introduction:
\begin{eqnarray}\label{new2}
Z_{V^+}((u, z_1), (v, z_2), \tau):=\Tr_{V^+}Y_V(q_1^{L(0)}u, q_1)Y_V(q_2^{L(0)}v, q_2)q^{L(0)-c/24}
\end{eqnarray}

\medskip
One of the needed adjustments is facilitated by the 
 following elementary observation from linear algebra.
\begin{lem}\label{lemspectrace} Let the finite-dimensional  linear space $X=X_1\oplus X_2$ decompose as indicated,
and let $f, g\in End(X)$ be a pair of endomorphisms mapping $X_1\rightarrow X_2$ and $X_2\rightarrow X_1$.\ 
 The following equivalent statements hold: (a) $\Tr_{X_2}fg= \Tr_{X_1}gf$;  (b) $\Tr_{X_1}(fg+gf)=\Tr_X fg$.\ In particular, (c) if $\Tr_X fg=0$ then
 $\Tr_{X_1}fg=-\Tr_{X_1}gf$.
\end{lem}
\begin{proof} Part  (a) is standard.\ As for  (b),  using (a) we have $\Tr_{X_1}(fg+gf)=\Tr_{X_1}fg+\Tr_{X_2}fg=\Tr_X fg$.
\end{proof}
\subsection{Statement and proof of  $\ZZ_2$-twisted Zhu reduction formula}\label{SSZRreduc}
We have the following statement of our modified $2$-point functions from the previous section,  which is to be viewed as a direct analogue of Theorem~\ref{thm1}.
\begin{thm}\label{thmZ22}($\ZZ_2$-twisted Zhu reduction) Let $u,v \in V^-$.\ 
The $2$-point function satisfies
\begin{eqnarray*}
&&\hspace{3cm} F_{V^+}((u, z_1), (v, z_2), \tau)=\\
&&=\frac{1}{2}\left\{ \Tr_V o(u)o(v)q^{L(0)-c/24}\right\}+\sum_{m\geq0} (-1)^{m+1}\frac{1}{m!}Q_1^{(m)}(z_{12}, \tau)Z_{V^+}(u[m]v, \tau) \\
&&+\frac 12 \sum_{m\geq 1}(-1)^{m+1}\frac{1}{m!} \left( P_1^{(m)}(z_{12}, \tau) -Q_1^{(m)}(z_{12}, \tau)\right) Z_V(u[m]v, \tau).
\end{eqnarray*}
\end{thm}
\begin{proof} 
We  start out as in the proof of Theorem \ref{thm1},  with $u\in V_k^-$ but using our modified $2$-point function:
\begin{eqnarray*}
F_{V^+}((u, z_1), (v, z_2), \tau):= \sum_{n\in\ZZ} q_1^{k-n-1} \Tr_{V^+} \left\{u(n)Y_{V}(q_2^{L(0)}v, q_2)\right\}q^{L(0)-c/24}.
\end{eqnarray*}

The next commutator formula is unchanged:
\begin{eqnarray*}
&&[u(n), Y_{V}(q_2^{L(0)}v, q_2)]=q_2^{n-k+1}\sum_{m\geq0} \frac{(n-k+1)^m}{m!}Y_{V}(q_2^{L(0)}u[m]v, q_2).
\end{eqnarray*}

Now we apply this formula and trace over $V^+$ to find that if $r:=n-k+1$ then
\begin{eqnarray*}
&&\ \ \ \ \ \ \ \ \ \ \ \ \  \Tr_{V^+}  \left(u(n)Y_V(q_2^{L(0)}v, q_2) \right)q^{L(0)-c/24} \\
&&=\Tr_{V^+} q_2^{r}\sum_{m\geq0} \frac{r^m}{m!}Y_{V}(q_2^{L(0)}u[m]v, q_2) q^{L(0)-c/24} + \Tr_{V^+}  \left(Y_{V}(q_2^{L(0)}v, q_2)u(n) q^{L(0)-c/24} \right)\\
&&= q_2^{r}\sum_{m\geq0} \frac{r^m}{m!}Z_{V^+}(u[m]v, \tau) + q^r \Tr_{V^+} \left(Y_{V}(q_2^{L(0)}v, q_2)q^{L(0)-c/24}u(n) \right)
\end{eqnarray*}
where we used (\ref{conjform}) to obtain the last term.\ Thanks to Lemma \ref{lemspectrace}(b)   we find that
\begin{eqnarray*}
&&\ \ \ \ \ \ \ \ \ \ \ \ \  \Tr_{V^+}  \left(u(n)Y_V(q_2^{L(0)}v, q_2) \right)q^{L(0)-c/24} \\
&&= q_2^{r}\sum_{m\geq0} \frac{r^m}{m!}Z_{V^+}(u[m]v, \tau) - q^r \Tr_{V^+} \left(u(n)Y_{V}(q_2^{L(0)}v, q_2)q^{L(0)-c/24} \right)\\
&&\ \ \ +q^r \Tr_{V} \left(Y_{V}(q_2^{L(0)}v, q_2)q^{L(0)-c/24}u(n) \right),
\end{eqnarray*}
and therefore
\begin{eqnarray*}
&&\ \ \ \ \ \ \ \ \ \ \ \ \  \Tr_{V^+}  \left(u(n)Y_V(q_2^{L(0)}v, q_2) \right)q^{L(0)-c/24} \\
&&= \frac{q_2^r}{1+q^r}\sum_{m\geq0} \frac{r^m}{m!}Z_{V^+}(u[m]v, \tau) +\frac{q^r}{1+q^r} \Tr_{V} \left(Y_{V}(q_2^{L(0)}v, q_2)q^{L(0)-c/24}u(n) \right).
\end{eqnarray*}

Now we use  parallel formulas derived in the course of the untwisted Zhu recursion, specifically (\ref{formular}), to see that
\begin{eqnarray*}
&&\ \ \ \ \ \ \ \ \ \ \ \ \  \Tr_{V^+}  \left(u(n)Y_V(q_2^{L(0)}v, q_2) \right)q^{L(0)-c/24} \\
&&= \frac{q_2^r}{1+q^r}\sum_{m\geq0} \frac{r^m}{m!}Z_{V^+}(u[m]v, \tau) \\
&& +\frac{1}{1+q^r}\left\{  \Tr_{V} \left(u(n)Y_{V}(q_2^{L(0)}v, q_2)q^{L(0)-c/24} \right)- q_2^r\sum_{m\geq0} \frac{r^m}{m!}Z_V(u[m]v, \tau)\right\}.
\end{eqnarray*}
Thus we obtain
\begin{eqnarray*}
&&\ \ \ \ \Tr_{V^+}  \left(u(n)Y_{V}(q_2^{L(0)}v, q_2) \right)q^{L(0)-c/24} \\
&&= \frac{q_2^r}{1+q^r}\sum_{m\geq0} \frac{r^m}{m!}\left\{  Z_{V^+}(u[m]v, \tau)-Z_V(u[m]v, \tau)\right\} \\
&&\ \ \  +\frac{1}{1+q^r}\left\{  \Tr_{V} \left(u(n)Y_{V}(q_2^{L(0)}v, q_2)q^{L(0)-c/24} \right)\right\}.
\end{eqnarray*}

The case $r=0$ tells us that
\begin{eqnarray}\label{formr=0}
&&\ \ \ \ \ \ \ \ \ \ \ \ \  \Tr_{V^+} o(u)o(v)q^{L(0)-c/24} \\
&&= \frac{1}{2}\left\{  Z_{V^+}(u[0]v, \tau)+ \Tr_{V} o(u)o(v)q^{L(0)-c/24} \right\}, \notag
\end{eqnarray}
where we also used  (\ref{[0]mode}).\ On the other hand if $r\neq0$ then the untwisted case shows that
\begin{eqnarray*}
&&\ \ \ \ \ \ \ \ \ \ \ \ \  \Tr_{V^+}  \left(u(n)Y_V(q_2^{L(0)}v, q_2) \right)q^{L(0)-c/24} \\
&&= \frac{q_2^r(1-q^r)}{1-q^{2r}}\sum_{m\geq0} \frac{r^m}{m!}\left\{  Z_{V^+}(u[m]v, \tau)-Z_V(u[m]v, \tau)\right\} \\
&&\ \ \  +\left\{ \frac{q_2^r}{1-q^{2r}}\sum_{m\geq0}\frac{r^m}{m!}Z_V(u[m]v, \tau) \right\}\\
&&=\frac{q_2^r}{1+q^{r}}  \left\{ \sum_{m\geq0} \frac{r^m}{m!}   Z_{V^+}(u[m]v, \tau) \right\}  + \frac{q_2^rq^r}{1-q^{2r}}  \sum_{m\geq 1}\frac{r^m}{m!}Z_V(u[m]v, \tau),
\end{eqnarray*}
where we used \eqref{[0]mode} once more to eliminate the $m=0$ term in the final sum.

Now we return to our $2$-point function to find that 
\begin{eqnarray*}
&&F_{V^+}((u, z_1), (v, z_2), \tau)= \sum_{n\in\ZZ} q_1^{k-n-1} \Tr_{V^+} \left\{u(n)Y_{V}(q_2^{L(0)}v, q_2)\right\}q^{L(0)-c/24} \\
&&=  \sum_{0\neq r\in\ZZ} \frac{q_{21}^r}{1+q^{r}}  \left\{ \sum_{m\geq0} \frac{r^m}{m!}   Z_{V^+}(u[m]v, \tau) \right\}  + \sum_{0\neq r\in\ZZ} \frac{q_{21}^rq^r}{1-q^{2r}} \sum_{m\geq 1} \frac{r^m}{m!} Z_V(u[m]v, \tau) \\
&&\ \ \ +\Tr_{V^+} o(u)o(v)q^{L(0) - c/24}\\
&&=\sum_{m\geq0} \frac{1}{m!}Q_1^{(m)}(z_{21}, \tau)Z_{V^+}(u[m]v, \tau) -\frac{1}{2}Z_{V^+}(u[0]v, \tau)  +\Tr_{V^+} o(u)o(v)q^{L(0) - c/24}\\
&&+\frac{1}{2}\sum_{0\neq r\in\ZZ} q_{21}^r\left( \frac{1}{1-q^{r}} -\frac{1}{1+q^r} \right) \sum_{m\geq 1} \frac{r^m}{m!} Z_V(u[m]v, \tau).
\end{eqnarray*}

Now use (\ref{formr=0}) to obtain
\begin{eqnarray*}
&&F_{V^+}((u, z_1), (v, z_2), \tau)= \\
&&\sum_{m\geq0} \frac{1}{m!}Q_1^{(m)}(z_{21}, \tau)Z_{V^+}(u[m]v, \tau) +\frac{1}{2}\left\{ \Tr_V o(u)o(v)q^{L(0)-c/24}  \right\}\\
&&+\frac{1}{2}\sum_{m\geq 1} \sum_{0\neq r\in\ZZ} q_{21}^r\left( \frac{1}{1-q^{r}} -\frac{1}{1+q^r} \right)  \frac{r^m}{m!} Z_V(u[m]v, \tau) \\
&&=\sum_{m\geq0} \frac{1}{m!}Q_1^{(m)}(z_{21}, \tau)Z_{V^+}(u[m]v, \tau) +\frac{1}{2}\left\{ \Tr_V o(u)o(v)q^{L(0)-c/24}\right\}\\
&&+\frac{1}{2}\sum_{m\geq 1}\frac{1}{m!} \left( P_1^{(m)}(z_{21}, \tau) -Q_1^{(m)}(z_{21}, \tau)\right) Z_V(u[m]v, \tau) \\
&&=\sum_{m\geq0} (-1)^{m+1}\frac{1}{m!}Q_1^{(m)}(z_{12}, \tau)Z_{V^+}(u[m]v, \tau) +\frac{1}{2}\left\{ \Tr_V o(u)o(v)q^{L(0)-c/24} \right\}\\
&&+\frac{1}{2}\sum_{m\geq 1}(-1)^{m+1}\frac{1}{m!} \left( P_1^{(m)}(z_{12}, \tau) -Q_1^{(m)}(z_{12}, \tau)\right) Z_V(u[m]v, \tau).
\end{eqnarray*}
In the above we also used Lemma \ref{lemPQ1}(b).\ This completes the proof of the Theorem.
\end{proof}

\section{$\ZZ_2$-twisted Zhu recursion}
\begin{thm}\label{thmZ2Z}($\ZZ_2$-twisted Zhu recursion.) In the setting of (\ref{setup}) with $u, v\in V^{-}$ and $n\geq1$,  we have
\begin{eqnarray*}
&& Z_{V^+}(u[-n]v, \tau)= \tfrac{1}{2}\delta_{n, 1} \Tr_V o(u)o(v)q^{L(0)-c/24}\\
&&+\sum_{m\geq0} \frac{1}{m}\widehat{F}_{n+m}(\tau))Z_{V^+}(u[m]v, \tau) +\frac 12\sum_{m\geq 1} \frac{1}{m}\left(\widehat{E}_{n+m}(\tau)-\widehat{F}_{n+m}(\tau) \right)  Z_V(u[m]v, \tau).
\end{eqnarray*}
\end{thm}
\begin{rmk}\label{remm=0} The term corresponding to $m=0$ in the first sum is equal to
\begin{eqnarray*}
 -F_{n}(\tau)Z_{V^+}(u[0]v, \tau).
\end{eqnarray*}
\end{rmk}

We shall use the next result later.\ It is a useful special case of Theorem \ref{thmZ2Z}.\ As for its proof, one only has to check the vanishing of the relevant terms in the recursive formula of Theorem \ref{thmZ2Z}, and this follows easily from the hypotheses together with (\ref{v[m]def}).

\begin{cor}\label{thmZ21} Let the notation and assumptions be as in Theorem \ref{thmZ2Z} and assume in addition that
$\Tr_V u(r)v(s)q^{L(0)-c/24}=0$ for all $r, s\in\ZZ$.\ Then 
\begin{eqnarray*}
&& Z_{V^+}(u[-n]v, \tau)=\sum_{m\geq0} \frac{1}{m}\widehat{F}_{n+m}(\tau))Z_{V^+}(u[m]v, \tau). 
\end{eqnarray*}
$\hfill\Box$
\end{cor}
\noindent
\textit{Proof of Theorem \ref{thmZ2Z}}.\ The first equality below continues to not only make sense for our modified $2$-point function, but it holds with the same proof (cf.\ \cite{MT2}), so we obtain
from Theorem \ref{thmZ22} that
\begin{eqnarray*}
&& Z_{V^+}(Y[u, z_{12}]v, \tau) = F_{V^+}((u, z_1), (v, z_2), \tau)\\
&&=\sum_{m\geq0} (-1)^{m+1}\frac{1}{m!}Q_1^{(m)}(z_{12}, \tau)Z_{V^+}(u[m]v, \tau) +\frac{1}{2} \Tr_V o(u)o(v)q^{L(0)-c/24}\\
&&\ \ \ \ \ +\sum_{m\geq1}(-1)^{m+1}\frac{1}{m!} \left( P_1^{(m)}(z_{12}, \tau) -Q_1^{(m)}(z_{12}, \tau)\right) Z_V(u[m]v, \tau).
\end{eqnarray*}

We identify powers of $z_{12}^{n-1}$.\ For $n\geq2$ we obtain
\begin{eqnarray*}
&&Z_{V^+}(u[-n]v, \tau)= \sum_{m\geq0} (-1)^{m+1}{n+m-1\choose m}F_{n+m}(\tau)Z_{V^+}(u[m]v, \tau) \\
&&+\sum_{m\geq 1}(-1)^{m+1} {n+m-1\choose m} \left( E_{n+m}(\tau) -F_{n+m}(\tau)\right) Z_V(u[m]v, \tau)
\end{eqnarray*}

On the other hand if $n=1$, noting that $P_1(z, \tau)$ has constant term zero in its Laurent expansion, we obtain
\begin{eqnarray*}
&&\tfrac{1}{2} \Tr_V o(u)o(v)q^{L(0)-c/24}+\sum_{m\geq1} (-1)^{m+1}\frac{1}{m!}F_{1+m}(\tau)Z_{V^+}(u[m]v, \tau)\\
&&+\sum_{m\geq1}(-1)^{m+1}   \left( E_{1+m}(\tau) - F_{1+m}(\tau) \right) Z_V(u[m] v,  \tau).
\end{eqnarray*}
Theorem \ref{thmZ2Z} is now proved. $\hfill\Box$
 
\section{Calculation of some $1$-point functions}\label{sec1point}
In this Section we solve Problem (c) of the Introduction for the VOAs $M^+$ and $V_L^+$ by combining both the $\ZZ_2$-twisted and untwisted Zhu recursion formulas.
\subsection{Symmetrized Heisenberg VOA}\label{sec1pointM}
As a first application of Theorem~\ref{thmZ2Z} we derive a formula for the 1-point functions for a basis of the symmetrized Heisenberg algebra $M^+$ (see \eqref{evenh}). 

\begin{thm}\label{thmM+} For positive integers $n_1, ..., n_p$ ($p$ even) we have
\begin{eqnarray*}
&&\ \ \ \ \ \ \ \ \ \ \ \  Z_{M^+}(h_{i_1}[-n_1]...h_{i_p}[-n_p]\mathbf{1}, \tau)=\\
&& \left(\sum_{\sigma} \prod_{(rs)} \delta_{i_r,i_s}\widehat{F}_{n_r+n_s}(\tau)\right)Z_{M^+}(\tau)+\\
&&\frac 12 \left(\sum_{\sigma} \left( \prod_{(r s)}\delta_{i_r,i_s}\widehat{E}_{n_r+n_s}(\tau) - \prod_{(r s)}\delta_{i_r,i_s}\widehat{F}_{n_r+n_s}(\tau)\right)\right)Z_M(\tau).
\end{eqnarray*}
where $\sigma$ ranges over the fixed-point free involutions of the index set $\{1,....,p\}$ and the products range over transpositions $(r s)$  whose product is $\sigma$. \end{thm}
\begin{proof}
As it plays no further r\^{o}le in the proof, we omit the argument $\tau$ from the notation for the sake of convenience.

\medskip
For $p=0$, using the convention that the identity is a fixed-point free involution on the empty set and an empty product equals $1$, the claimed formula in the Theorem yields
$$Z_{M^+}=Z_{M^+}+\frac 12 (1-1)Z_M,$$
which is clearly true.

\medskip
Supposing the claim is true for some even $p\geq 2$, we now show it for $p+2$.  So let $u=h_{i_1}[-n_1]\dots h_{i_{p+2}}[-n_{p+2}]\mathbf 1$ and write $v=h_{i_2}[-n_2]\dots h_{i_{p+2}}[-n_{p+2}]\mathbf 1$. Then Theorem~\ref{thmZ22} yields
\begin{align*}
Z_{M^+}(u)=&\frac{1}{2}\delta_{n_1,1}\Tr_M o(h_{i_1})o(v) q^{L(0)-c/24}\\
&+\sum_{m\geq 0} \frac 1m\widehat F_{n_1+m}Z_{M^+}(h_{i_1}[m] v)\\
&+\frac 12 \sum_{m\geq 1} \frac 1m\left(\widehat E_{n_1+m}-\widehat F_{n_1+m}\right)Z_{M}(h_{i_1}[m] v),
\end{align*}
which, using \eqref{commrelns}, we may again simplify to
\begin{gather}\label{eqMplusinduction}
\sum_{j=2}^{p+2} \delta_{i_1,i_j}\widehat F_{n_1+n_j}Z_{M^+}(v\setminus j)
+\frac 12 \sum_{j=2}^{p+2} \delta_{i_1,i_j}\left(\widehat E_{n_1+n_j}-\widehat F_{n_1+n_i}\right)Z_{M}(v\setminus j),
\end{gather}
where we write $v\setminus j$ for the state obtained from $v$ by deleting the operator $h_{i_j}[-n_j]$. By the induction hypothesis we have that 
$$Z_{M^+}(v\setminus j)=\left(\sum_{\sigma_j} \prod_{(rs)}  \delta_{i_r,i_s}\widehat{F}_{n_r+n_s}\right)Z_{M^+}
+\frac 12 \left(\sum_{\sigma_j} \left( \prod_{(rs)}\delta_{i_r,i_s}\widehat{E}_{n_r+n_s}- \prod_{(rs)}\delta_{i_r,i_s}\widehat{F}_{n_r+n_s}\right)\right)Z_M,$$
where $\sigma_j$ ranges over the fixed-point free involutions of the index set $\{2,...,p+2\}\}\setminus\{j\}$.\ By Theorem~\ref{thmMT}, or rather its more general version for arbitrary ranks (see \cite[Theorem 3.14]{MT2}) on the other hand, we find that 
$$Z_{M}(v\setminus j)=\left(\sum_{\sigma_j} \prod_{(rs)} \delta_{i_r,i_s}\widehat{E}_{n_r+n_s}\right)Z_{M}.$$
Plugging this into \eqref{eqMplusinduction} yields the following expression for $Z_{M^+}(u)$,
\begin{align*}
 &\sum_{j=2}^{p+2}  \left(\sum_{\sigma_j} \delta_{i_1,n_j}\widehat F_{n_1+n_i}\prod_{(rs)}\widehat \delta_{i_r,i_s}F_{n_r+n_s}\right)Z_{M^+}\\
&+\frac{1}{2}\sum_{j=2}^{p+2}  \left(\sum_{\sigma_j} \delta_{i_1,i_j}\widehat F_{n_1+n_i}\left(\prod_{(rs)}\delta_{i_r,i_s}\widehat E_{n_r+n_s}-\prod_{(rs)}\delta_{i_r,i_s}\widehat F_{n_r+n_s}\right)\right)Z_{M}\\
&+\frac{1}{2}\sum_{j=2}^{p+2} \left(\sum_{\sigma_j} \delta_{i_1,i_j}(\widehat E_{n_1+n_j}-\widehat F_{n_1+n_j})\prod_{(rs)}\delta_{i_r,i_s}\widehat E_{n_r+n_s}\right)Z_{M}
\end{align*}
We can clearly simplify the first line to 
$$\left(\sum_{\sigma} \prod_{(rs)} \delta_{i_r,i_s}\widehat F_{n_r+n_s}\right)Z_{M^+},$$
where $\sigma$ now ranges over all fixed-point free involutions of the full index set $\{1,...,{p+2}\}$.  The second and third line together simplify to
$$\frac{1}{2}\sum_{j=2}^{p+2} \left(\sum_{\sigma_j}  \delta_{i_1,i_j}\widehat E_{n_1+n_j}\prod_{(rs)}\delta_{i_r,i_s}\widehat E_{n_r+n_s}-\delta_{i_1,i_j}\widehat F_{n_1+n_i}\prod_{(rs)}\delta_{i_r,i_s}\widehat F_{n_r+n_s}\right)Z_M,$$
yielding
$$\frac 12\left(\sum_{\sigma}\left(\prod_{(rs)} \delta_{i_r,i_s}\widehat E_{n_r+n_s}- \prod_{(rs)}\delta_{i_r,i_s} \widehat F_{n_r+n_s}\right)\right)Z_{M}$$
with the range of $\sigma$ as before. This completes the proof of the Theorem.
\end{proof}

\subsection{Symmetrized lattice VOAs}\label{secsymlat} In this Subsection we compute all trace functions in a symmetrized lattice theory $V_L^+$.\ 

\medskip
Let $M\subseteq V_L$ be the usual Heisenberg subVOA.\ $V_L^+$ is spanned by the states described in (\ref{evenh}) and (\ref{genstates}).\ Concerning the latter types of states,
consider a state of the form $h_{i_1}[-n_1]...h_{i_p}[-n_p]f_{\alpha}$\ ($p\geq 2$ even) and take $u=h, v=h_{i_2}[-n_2]...h_{i_p}[-n_p]f_{\alpha}$.\ Setting $M_{\alpha}=M\otimes e^{\alpha}$ for $\alpha\in L$,  these are  $M$-modules with fusion rules are $M_{\alpha}\boxtimes M_{\beta}=M_{\alpha+\beta}$.\ It follows immediately that the hypotheses of Corollary \ref{thmZ21} are satisfied with
$V=V_L$.\ Thus we obtain the recursion
\begin{eqnarray}\label{lattrecur}
&&\ \ \ \ \ \ \ \ \ \  Z_{V_L^+}(h_{i_1}[-n_1]...h_{i_p}[-n_p]f_{\alpha}, \tau)\\
&&=\sum_{m\geq0} \frac{1}{m}\widehat{F}_{n+m}(\tau))Z_{V_L^+}(h_{i_1}[m]h_{i_2}[-n_2]...h_{i_p}[-n_p]f_{\alpha}, \tau).\notag
\end{eqnarray}

Next we will explain the solution of this recursive formula.\
For each positive integer $n$ we introduce the modular expression 
\begin{eqnarray*}
\widetilde{F}_{n}(\tau):= -(h_{n}, \alpha)F_{n}(\tau).
\end{eqnarray*}
All of this applies with only cosmetic changes if we replace $f_{\alpha}$ with $g_{\alpha}$, and we carry over the preceding notations in this case.\
Finally we can state

\begin{thm}\label{thmlattrecur}
Let  $0\neq\alpha\in L$ and $u:=h_{i_1}[-n_1] ...h_{i_p}[-n_p]e_{\alpha}$,  where $e_{\alpha}=f_{\alpha}$ or $g_{\alpha}$ according as
$p$ is even or odd respectively (cf. (\ref{fgdef})).\  Then
\begin{eqnarray*}
Z_{V_L^+}(u, \tau)= \left\{\sum_{\sigma} \prod_{(rs),(t)} \widetilde{F}_{n_t}(\tau)\widehat{F}_{n_r+n_s}(\tau) \right\} Z_{V_L^+}(f_{\alpha}, \tau)
\end{eqnarray*}
where $\sigma$ ranges over all involutions of the index set $\{1,...,p\}$, and $(rs),(t)$ range over the $2$-cycles and $1$-cycles respectively in $S_p$ whose product is $\sigma$.
\end{thm} 
\begin{proof} \ We proceed by induction on $p$, the case $p=0$ being trivially true.

\medskip
Use  (\ref{lattrecur}) together with the commutator relations (\ref{commrelns}) to see that (with $p$ even, say)
\begin{eqnarray}\label{newform}
&&\ \ \ \ \ \ \ \ \ \ \ Z_{V_L^+}(h_{i_1}[-n_{1}]...h_{i_p}[-n_{p}]f_{\alpha}, \tau)\\
&&= \sum_{m\geq0} \frac{1}{m}\widehat{F}_{n_1+m}(\tau))Z_{V_L^+}(h_{i_1}[m]h_{i_2}[-n_2]...h_{i_p}[-n_p]f_{\alpha}, \tau) \notag
\end{eqnarray}
where, using Remark \ref{remm=0}, the term with $m=0$ is equal to\ 
\begin{eqnarray*}
 && -(h_{i_1}, \alpha)F_{n_1}(\tau)Z_{V_L^+}(h_{i_2}[-n_2]...h_{i_p}[-n_p]g_{\alpha}, \tau)\\
  &&=  -(h_{i_1}, \alpha)F_{n_1}(\tau)  \left\{\sum_{\sigma'}  \prod_{(rs),(t)} \widetilde{F}_{n_t}(\tau)\widehat{F}_{n_r+n_s}(\tau) \right\} Z_{V^+}(f_{\alpha}, \tau)
\end{eqnarray*}
where $\sigma'$ ranges over involutive permutations of the symmetric group on the index set  $\{2,...,p\}$.\ Thus the term $m=0$
accounts for all of  the  permutations in $S_p$  having square $1$ and $1$ as a fixed point.

\medskip
Similarly, the terms with $m\neq0$ on the right-hand side of (\ref{newform}) yield
\begin{eqnarray*}
&&\sum_{j=2}^p \frac{1}{n_j}\widehat{F}_{n_1+n_j}(\tau) Z_{V^+}(h[n_j]h[-n_2]...h[-n_k]g_{\alpha}, \tau)\\
&&=\sum_{j=2}^p \widehat{F}_{n_1+n_j}(\tau)   \left\{\sum_{\sigma_{1, j}}  \prod_{(rs),(t) } \widetilde{F}_{n_t}(\tau)\widehat{F}_{n_r+n_s}(\tau) \right\} Z_{V^+}(f_{\alpha}, \tau),
\end{eqnarray*}
where $\sigma_{1, j}$ ranges over the involutive permutations of the symmetric group on the  index set $\{1, ..., p\}\setminus\{1, j\}$
Thus the terms with $m\neq 0$ account for all involutive permutations in $S_p$ for which $1$ is not a fixed point. \
This proves the Theorem when $p$ is even.\ If $p$ is odd the proof is similar with only changes of notation and we skip the details.\  The Theorem is now proved.
\end{proof}

\begin{rmk}\label{rmkfalphatrace}
The formula in Theorem~\ref{thmlattrecur} involves the trace of the state $f_\alpha$ on $V_L^+$.  It is not hard to see that this trace vanishes unless $\alpha=2\beta\in 2L$. In this case the only states that can contribute to the trace lie in $(M^+\otimes f_\beta)\oplus (M^-\otimes g_\beta)$. An explicit formula can be found in \cite[Lemma 4.3]{DM} in the special case of $L=\Lambda$ the Leech lattice, but the proof applies for general even lattices $L$ of rank $k$:
\begin{gather}
Z_{V_L^+}(f_{\alpha},\tau)=\frac{\eta(2\tau)^{2(\alpha,\alpha)-k}}{\eta(\tau)^{(\alpha,\alpha)-k}}.
\end{gather}
\end{rmk}

It remains to compute the traces of states of the form $h_{i_1}[-n_1]...h_{i_p}[-n_{p}]\mathbf 1\in M^+\subset V_L^+$ for $r$ even (see \eqref{evenh}).  The idea is once more to employ $\ZZ_2$-twisted Zhu recursion from Theorem~\ref{thmZ2Z}.  For this we require the following lemma on traces in the full lattice VOA $V_L$. It is essentially a special case of \cite[Theorem 3.14]{MT2} for $n=0$. Nevertheless, we give a proof here as well for the convenience of the reader.

\medskip
Before stating the lemma we introduce the following notation. 
For a finite set $A$ we let $S(A)$ be the symmetric group acting on $A$ and we abbreviate the set $\{a,a+1,...,b\}$ ($a,b\in\ZZ_{>0}$) by $\underline{ab}$ or simply $\underline b$ if $a=1$. Further let 
\begin{gather}
\Inv(A):=\{\sigma\in S(A)\: :\: \sigma^2=\id_A\}
\end{gather}
denote the set of all involutions in $S(A)$ and 
\begin{gather}
\Inv_0(A):=\{\sigma\in\Inv(A)\: :\: \sigma(x)\neq x\text{ for all }x\in A\}
\end{gather}
the set of fixed-point free involutions on $A$. 

We also require the following variant of the theta function of a lattice $L$.  For any function $P:L\to\CC$ (which usually is a polynomial) set
\begin{gather}
\theta_L(\tau;P)=\sum_{\alpha\in L}P(\alpha)q^{(\alpha,\alpha)/2}.
\end{gather}
Note that if $P(-\alpha)=-P(\alpha)$, then $\theta_L(\tau,P)$ vanishes identically.
\begin{lem}\label{lemVL}
Consider the $M$-module $N:=M\otimes e^{\alpha}\subset V_L$ for some $\alpha\in L$ and let $u= h_{i_1}[-n_1]...h_{i_p}[-n_p]\mathbf 1$ for \emph{any} $p\geq 1$.  Further let $\Lambda=\{j\in\underline p\: :\: n_j=1\}.$ Then we have
\begin{align}
\label{eqZNu}Z_N(u,\tau)&=\sum_{\Delta\subseteq\Lambda} \big(\prod_{j\in\Delta} (h_{i_j},\alpha)\big) \frac{q^{(\alpha,\alpha)/2}}{\eta(\tau)^k} \sum_{\sigma\in \Inv_0(\underline r\setminus \Delta)} \prod_{(rs)} \delta_{i_r,i_s}\widehat E_{n_r+n_s},\\
\label{eqZVLu}Z_{V_L}(u,\tau)&=\sum_{\Delta\subseteq\Lambda}  \frac{\theta_L(\tau, P_\Delta)}{\eta(\tau)^k} \sum_{\sigma\in \Inv_0(\underline r\setminus \Delta)} \prod_{(rs)} \delta_{i_r,i_s}\widehat E_{n_r+n_s},
\end{align}
where $(rs)$ runs through all transpositions whose product is $\sigma$ and we set
$$P_\Delta(\alpha):=\prod_{j\in\Delta}(h_{i_j},\alpha).$$
\end{lem}
\begin{proof}
We first prove \eqref{eqZNu}. For $p=1$ we find by the standard Zhu recursion (Theorem~\ref{thm2}) that
\begin{align*}
Z_N(u,\tau)&=\delta_{n_1,1}\Tr_N o(h_{i_1}) o(\mathbf 1)q^{L(0)-c/24}+\sum_{m\geq 1}\frac 1m \widehat E_{n_1+m} Z_N(h_{i_1}[m]\mathbf 1)\\
&=\delta_{n_1,1} (h_{i_1},\alpha)\frac{q^{(\alpha,\alpha)/2}}{\eta^k}
\end{align*}
by \eqref{eqhnealpha} together with the standard fact that $o(\mathbf 1)=\id_V$ in any VOA $V$. This is precisely what \eqref{eqZNu} reduces to for $p=1$, using the convention that the identity is a fixed-point free involution on the empty set and an empty product is $1$ by definition.

Now suppose \eqref{eqZNu} holds for all $p'\leq p$ some $p\geq 1$ and. consider $u=h_{i_1}[-n_1]v$ with $v=h_{i_2}[-n_2]...h_{i_{p+1}}[-n_{p+1}]\mathbf 1$. Then we find, again by Zhu recursion and similar considerations as e.g. in the proof of Theorem~\ref{thmlattrecur} that
\begin{align*}
Z_N(u,\tau)&=\delta_{n_1,1}\Tr_N o(h_{i_1}) o(v)q^{L(0)-c/24}+\sum_{m\geq 1}\frac 1m \widehat E_{n_1+m} Z_N(h_{i_1}[m]v)\\
&=\delta_{n_1,1} (h_{i_1},\alpha) Z_N(v,\tau)+\sum_{j=2}^{p+1} \delta_{i_1,i_j}\widehat E_{n_1+n_j}Z_N(v\setminus j)\\
&=\delta_{n_1,1} (h_{i_1},\alpha) \left[\sum_{\Delta\subseteq\Lambda\setminus \{1\}} \big(\prod_{j\in\Delta} (h_{i_j},\alpha)\big) \frac{q^{(\alpha,\alpha)/2}}{\eta(\tau)^k} \sum_{\sigma\in \Inv_0(\underline {2,p+1}\setminus \Delta)} \prod_{(rs)} \delta_{i_r,i_s}\widehat E_{n_r+n_s}\right]\\
&\qquad +\sum_{j=2}^{p+1} \delta_{i_1,i_j}\widehat E_{n_1+n_j} \left[\sum_{\Delta\subseteq\Lambda\setminus\{1,j\}} \big(\prod_{j\in\Delta} (h_{i_j},\alpha)\big) \frac{q^{(\alpha,\alpha)/2}}{\eta(\tau)^k} \sum_{\sigma\in \Inv_0(\underline {p+1}\setminus \Delta\setminus\{1,j\})} \prod_{(rs)} \delta_{i_r,i_s}\widehat E_{n_r+n_s}\right]\\
&=\sum_{\substack{ \Delta\subseteq\Lambda\\ 1\in\Delta}} \big(\prod_{j\in\Delta} (h_{i_j},\alpha)\big) \frac{q^{(\alpha,\alpha)/2}}{\eta(\tau)^k} \sum_{\sigma\in \Inv_0(\underline {p+1}\setminus \Delta)} \prod_{(rs)} \delta_{i_r,i_s}\widehat E_{n_r+n_s}\\
&\qquad +\sum_{\substack{ \Delta\subseteq\Lambda\\ 1\notin\Delta}} \big(\prod_{j\in\Delta} (h_{i_j},\alpha)\big) \frac{q^{(\alpha,\alpha)/2}}{\eta(\tau)^k} \sum_{\sigma\in \Inv_0(\underline {p+1}\setminus \Delta)} \prod_{(rs)} \delta_{i_r,i_s}\widehat E_{n_r+n_s}\\
&=\sum_{\Delta\subseteq\Lambda} \big(\prod_{j\in\Delta} (h_{i_j},\alpha)\big) \frac{q^{(\alpha,\alpha)/2}}{\eta(\tau)^k} \sum_{\sigma\in \Inv_0(\underline {p+1}\setminus \Delta)} \prod_{(rs)} \delta_{i_r,i_s}\widehat E_{n_r+n_s}.
\end{align*}
Hence we have shown \eqref{eqZNu}.

Since $V_L=\bigoplus _{\alpha\in L} M\otimes e^\alpha$, we obtain \eqref{eqZVLu} from \eqref{eqZNu} simply by summing over $\alpha\in L$.
\end{proof}

With Lemma~\ref{lemVL} we can now proceed to proving the following theorem, which complements Theorem~\ref{thmlattrecur} in that it provides a formula for the trace function of the states $u\in M^+\subset V_L^+$.

\begin{thm}\label{thmVL+trace}
Let $u\in M^+\subseteq V_L^+$ be a state as given above (see also \eqref{evenh}) and let $\Lambda$ and $P_\Delta$ be as in Lemma~\ref{lemVL}.  Then we have that 
\begin{multline*}
Z_{V_L^+}(u,\tau)=\big(\sum_{\sigma\in \Inv_0(\underline p)} \big(\prod_{(rs)} \delta_{i_r,i_s}\widehat F_{n_r+n_s}(\tau)\big)\big) \left(Z_{V_L^+}(\tau)-\frac 12 Z_{V_L}(\tau)\right)\\
+\frac 12\sum_{\Delta\subseteq\Lambda} \frac{\theta_L(\tau,P_\Delta)}{\eta^k}\sum_{\sigma\in \Inv_0(\underline p\setminus\Delta)} \big(\prod_{(rs)} \delta_{i_r,i_s}\widehat E_{n_r+n_s}(\tau)\big).
\end{multline*}
\end{thm}
\begin{proof}
First note that for $p=0$, the claimed formula for $Z_{V_L^+}(u,\tau)$ yields
$$Z_{V_L^+}(\tau)-\frac 12 Z_{V_L}(\tau)+\frac 12 Z_{V_L}(\tau)=Z_{V_L^+}(\tau),$$
as it should.

\medskip
Now suppose the formula has been proved for all even $p'<p$ for some $p\geq 2$.  Writing again $u=h_{i_1}[-n_1]v$ with $v=h_{i_2}[n_2]...h_{i_p}[-n_p]\mathbf 1$, we use the $\ZZ_2$-twisted version of Zhu recursion from Theorem~\ref{thmZ2Z} to obtain, leaving out $\tau$ from the discussion once more to improve legibility and to save space,
\begin{align}
\nonumber Z(u)&=\frac 12\delta_{n_1,1} \Tr_{V_L}o(h_1)o(v)q^{L(0)-c/24}+\sum_{m\geq 0} \frac 1m \widehat F_{n_1+m}Z_{V_L^+}(h_{i_1}[m] u)\\
\nonumber &\qquad\qquad +\frac 12 \sum_{m\geq 1}\frac 1m \left(\widehat E_{n_1+m}-\widehat F_{n_1+m}\right)Z_{V_L}(h_{i_1}[m]v)\\
\label{eqintermed}&=\frac 12 \sum_{\alpha\in L}\delta_{n_1,1}(h_{i_1},\alpha)  Z_{M\otimes e^{\alpha}}(v)
+ \sum_{j=2}^p \widehat \delta_{i_1,i_j}F_{n_1+n_j}Z_{V_L^+}(v\setminus j)\\
\nonumber &\qquad\qquad +\frac 12 \sum_{j=2}^p\left( \delta_{i_1,i_j}\widehat E_{n_1+n_j}- \delta_{i_1,i_j}\widehat F_{n_1+n_j}\right)Z_{V_L}(v\setminus j).
\end{align} 
According to Lemma~\ref{lemVL} we have 
$$Z_{M\otimes e^{\alpha}}(v)=\sum_{\Delta\subseteq\Lambda\setminus 1} \big(\prod_{j\in\Delta} (h_{i_j},\alpha)\big) \frac{q^{(\alpha,\alpha)/2}}{\eta(\tau)^k} \sum_{\sigma\in \Inv_0(\underline p\setminus \Delta)} \prod_{(rs)} \delta_{i_r,i_s}\widehat E_{n_r+n_s}$$
as well as
$$Z_{V_L}(v\setminus j)=\sum_{\Delta\subseteq\Lambda\setminus\{1,j\}}  \frac{\theta_L(\tau, P_\Delta)}{\eta(\tau)^k} \sum_{\sigma\in \Inv_0(\underline p\setminus \Delta\setminus\{1,j\})} \prod_{(rs)} \delta_{i_r,i_s}\widehat E_{n_r+n_s},$$
while we have by the induction hypothesis that
\begin{multline*}
Z_{V_L^+}(v\setminus j)=\big(\sum_{\sigma\in \Inv_0(\underline {2,p}\setminus j)} \big(\prod_{(rs)} \delta_{i_r,i_s}\widehat F_{n_r+n_s}\big)\big) \left(Z_{V_L^+}-\frac 12 Z_{V_L}\right)\\
+\frac 12\sum_{\Delta\subseteq\Lambda\setminus\{1,j\}} \frac{\theta_L(\cdot,P_\Delta)}{\eta^k}\sum_{\sigma\in \Inv_0(\underline p\setminus\{1,j\}\setminus\Delta)} \big(\prod_{(rs)} \delta_{i_r,i_s}\widehat E_{n_r+n_s}\big)
\end{multline*}
Plugging this into \eqref{eqintermed} and simplifying then yields 
\begin{align*}
&\frac 12 \sum_{\alpha\in L}\delta_{n_1,1}(h_1,\alpha) \sum_{\Delta\subseteq\Lambda\setminus 1} P_\Delta(\alpha) \frac{q^{(\alpha,\alpha)/2}}{\eta(\tau)^k} \sum_{\sigma\in \Inv_0(\underline p\setminus \Delta)} \prod_{(rs)} \delta_{i_r,i_s}\widehat E_{n_r+n_s}\\
&\qquad\qquad + \sum_{j=2}^p \widehat \delta_{i_1,i_j}F_{n_1+n_j}\big(\sum_{\sigma\in\Inv_0(\underline{2,p}\setminus j)}\big(\prod_{(rs)}\delta_{i_r,i_s}\widehat F_{n_r+n_s}\big)\big)\left(Z_{V_L^+}-\frac 12Z_{V_L}\right)\\
&\qquad\qquad +\frac 12 \sum_{j=2}^p \delta_{i_1,i_j}\widehat E_{n_1+n_j}\sum_{\Delta\subseteq \Lambda\setminus\{1,j\}}\frac{\theta_L(\cdot,P_\Delta}{\eta^k}\sum_{\sigma\in\Inv_0(\underline p\setminus\Delta\setminus\{1,j\})}\big(\prod_{(rs)}\delta_{i_r,i_s}\widehat E_{n_r+n_s}\big)\\
=&\frac 12 \sum_{\substack{\Delta\subseteq\Lambda\\ 1\in\Delta}}\frac{\theta_L(\cdot,P_\Delta}{\eta^k}\sum_{\sigma\in\Inv_0(\underline p\setminus\Delta)}\big(\prod_{(rs)}\delta_{i_r,i_s}\widehat E_{n_r+n_s}\big)\\
&\qquad\qquad + \big(\sum_{\sigma\in\Inv_0(\underline{p})}\big(\prod_{(rs)}\delta_{i_r,i_s}\widehat F_{n_r+n_s}\big)\big)\left(Z_{V_L^+}-\frac 12Z_{V_L}\right)\\
&\qquad\qquad+\frac 12 \sum_{\substack{\Delta\subseteq\Lambda\\ 1\notin\Delta}}\frac{\theta_L(\cdot,P_\Delta}{\eta^k}\sum_{\sigma\in\Inv_0(\underline p\setminus\Delta)}\big(\prod_{(rs)}\delta_{i_r,i_s}\widehat E_{n_r+n_s}\big),
\end{align*}
which simplifies further to the claimed formula. This completes the proof.
\end{proof}

We conclude by analyzing the modularity properties of $Z_{V_L^+}(u,\tau)$ as given in Theorem~\ref{thmVL+trace}.  The first line of the formula stated there is clearly a modular form of weight 
$$K=\sum_{j=1}^p n_j$$ 
for the group $\Gamma_0(\lcm(2,N))$ with some multiplier system, where $N$ is the level of the underlying lattice by Schoeneberg's Theorem~\ref{thmSchoeneberg}. So we only need to consider the second line of the formula,
\begin{gather}\label{eqmodular}
G(u, \tau)=\sum_{\Delta\subseteq\Lambda} \frac{\theta_L(\tau,P_\Delta)}{\eta^k(\tau)}\sum_{\sigma\in \Inv_0(\underline p\setminus\Delta)} \big(\prod_{(rs)} \delta_{i_r,i_s}\widehat E_{n_r+n_s}(\tau)\big).
\end{gather}

The remainder of this section is devoted to proving the following result, which gives an alternative proof and a more explicit formulation for the modularity of the $1$-point functions $Z_{V_L^+}(u)$.

\begin{prop}\label{propmod}
For a state $u\in V_{L}^+$ as in Theorem~\ref{thmVL+trace} the function $G(u)$ defined in \eqref{eqmodular} is modular of weight $K$ with respect to the group $\Gamma_0(N)$, where $N$ is the level of the lattice $L$ with a certain multiplier system.
\end{prop}
\begin{proof}
For simplicity we assume from now on that the rank of the underlying lattice $L$ is $k=1$ and we leave the details for the higher rank case to the reader.\ Therefore we may assume 
$$u=h[-n_1] h[-n_2]...h[-n_p]\mathbf 1,$$
where we assume that exactly the first $\ell$ of the $n_j$ are $1$, so in the notation of Lemma~\ref{lemVL} $\Lambda=\{1,...,\ell\}.$ 
In this case we find that $P_\Delta(\alpha)=(h,\alpha)^{\# \Delta}$, so that we can write 
$$\theta_L(\tau,P_\Delta)=\theta_L(\tau,h,\#\Delta)$$
with $\theta_L(\tau,v,m)$ as in \eqref{eqthetavm} and the $\delta$-functions in \eqref{eqmodular} all become $1$.  

We assume further that $\ell=2\ell'$ is even, the case $\ell=2\ell'+1$ odd can be handled analogously. Using these simplifications we may rewrite $G(u,\tau)$ as
\begin{align}
\eta(\tau)G(u,\tau)&=\sum_{m=0}^{\ell'} \binom{2\ell'}{2m} \theta_L(\tau,h,2m)\sum_{\sigma\in\Inv_0(\underline p\setminus\underline{2m})} \big(\prod_{(rs)}\widehat E_{n_r+n_s}(\tau)\big)\label{eqstep1} \\
&=\sum_{m=0}^{\ell'} \binom{2\ell'}{2m} \theta_L(\tau,h,2m)\sum_{n=0}^{\ell'-m}\frac{(2n)!}{n!2^n}\binom{2(\ell'-m)}{2n}E_2(\tau)^n\label{eqstep2}\\
&\qquad\qquad\qquad \qquad \qquad \qquad \qquad \qquad \times \sum_{\substack{\sigma\in\Inv_0(\underline p\setminus\underline{2(m+n)})\notag\\ \sigma(i)>2\ell', i\leq 2\ell}} \big(\prod_{(rs)}\widehat E_{n_r+n_s}(\tau)\big)\notag\\
&=\frac{(2\ell')!}{2^{\ell'}}\sum_{m=0}^{\ell'} \frac{2^m}{(2m)!} \theta_L(\tau,h,2m)\sum_{n=0}^{\ell'-m}\frac{E_2(\tau)^n}{n!}\label{eqG}\\
&\qquad\qquad\qquad \qquad \qquad \qquad \times \frac{2^{\ell'-m-n}}{(2(\ell'-m-n))!} \sum_{\substack{\sigma\in\Inv_0(\underline p\setminus\underline{2(m+n)})\\ \sigma(i)>2\ell', i\leq 2\ell}} \big(\prod_{(rs)}\widehat E_{n_r+n_s}(\tau)\big)\notag
\end{align}
where we  used in \eqref{eqstep1} that the summation over $\Delta$ in \eqref{eqmodular} depends only on the cardinality of $\Delta$ and that $\theta_L(\tau,h,m)=0$ for $m$ odd; in \eqref{eqstep2} we collected all possible powers of $E_2$ in the expressions involved; in \eqref{eqG} we simply reorganized the binomial and other coefficients. Note that the expression in the last line is indeed truly modular for the full modular group $\SL_2(\ZZ)$ of the correct weight since at least one of $n_r$ and $n_s$ is larger than $1$. 

\medskip
Now consider the Jacobi-like forms $\Theta_L(\tau,h,X)$ and $\widetilde E(\tau,X)$ as in \eqref{eqThetaX} and \eqref{eqEhat} and let $(f_m)_{m\geq 0}$ be any sequence of modular forms of weight $K+2m$ with respect to the full modular group $\SL_2(\ZZ)$ for some even integer $K>0$. Then the function
$$F(\tau,X):=\sum_{m=0} \frac{f_m(\tau)}{m!}(2\pi iX)^m$$
is clearly a Jacobi-like form of weight $K$ and index $0$.  Since we know from Theorem~\ref{thmSchoeneberg} and Lemma~\ref{lemEhat} that $\Theta_L(\tau,h,X)$ and $\widetilde E(\tau,X)$ are both Jacobi-like forms of index $1$ and weight $1/2$ resp.\ $0$ for some $\Gamma_0(N)$ with a certain multiplier system resp.\ Therefore the function $\Theta_L(\tau,h,X)\widetilde E(\tau,-X)F(\tau,X)$ is a Jacobi-like form of weight $K$ and index $0$ for $\Gamma_0(N)$ with the same multiplier system as $\Theta_L(\tau,h,X)$, so that the coefficient of $(2\pi iX)^\ell$ in it, which is given by
\begin{gather}\label{eqcoeffX}
\sum_{m=0}^\ell \sum_{n=0}^{\ell-m} \frac{2^m}{(2m)!}\theta_L(\tau,h,2m) \cdot \frac{E_2(\tau)^n}{n!} \cdot \frac{f_{\ell-m-n}(\tau)}{(\ell-m-n)!},
\end{gather}
is modular of weight $K+1/2+2\ell$ for the same group and multiplier system.

\medskip
Upon comparing \eqref{eqG} and \eqref{eqcoeffX} we see that they agree up to a constant factor, so it follows that the function $G$ from \eqref{eqmodular} indeed has the modularity properties claimed in the Proposition.
\end{proof}

\bibliographystyle{amsplain}

\end{document}